\documentclass{article}
\usepackage{amsmath,amsfonts}

\begin{document}

\makeatletter
\newenvironment{exemples}[1][\hskip-1ex]{\par\reset@font{\slshape %
    Exemples\hskip1ex#1. }}{\par}
\newenvironment{comp}[1][\hskip-1ex]{\par\reset@font{\slshape %
    Complement\hskip1ex#1\pointir }}{\par}
\newenvironment{proof}[1][\hskip-1ex]{\par\reset@font{\slshape %
    Proof\hskip1ex#1\pointir}}{\hfill\finpr\par}
    \newenvironment{sketchproof}[1][\hskip-1ex]{\par\reset@font{\slshape %
    Sketch of proof\hskip1ex#1\pointir}}{\hfill\finpr\par}
\newenvironment{remarks}{\par %
\noindent\textbf{Remarks}\pointir}{\hfill\par}
\newenvironment{remark}{\par %
\noindent\textbf{Remark}\pointir}{\hfill\par}  
\newenvironment{example}{\par %
\noindent\textbf{Example}}{\hfill\par} 
 
\newtheorem{corollary}{Corollary}[section]
\newtheorem{definition}{Definition}[section]
\newtheorem{question}{Question}[section]
\newtheorem{theo}{Theorem}[section]
\newtheorem{lemma}{Lemma}[section]
\newtheorem{proposition}{Proposition}[section]

\long\def\InsertFig#1 #2 #3 #4\EndFig{\hbox{\hskip #1 mm$
\vbox to #2mm{\vfil\includegraphics{#3}}#4$}}
\long\def
\LabelTeX#1 #2 #3\ELTX{\rlap{\kern#1mm\raise#2mm\hbox{#3}}}%

\def\build#1#2\fin{\mathrel{\mathop{\kern0pt#1}\limits#2}}
\def\pointir{.\kern.4em\ignorespaces}%
\def\openbox{\leavevmode
  \hbox to.77778em{%
  \hfil\vrule
  \vbox to.675em{\hrule width.6em\vfil\hrule}%
  \vrule\hfil}}
\def\finpr{\openbox}

\def\hfl{{\hbox to 8mm{\rightarrowfill}}}

\def\hf#1{\hhf^{n_{#1}}}
\def\sn#1{S^{n_{#1}{-}1}}
\newcommand{\xgg}{(\wt X,g^1_0\oplus g^2_0)}
\def\scr{\scriptstyle}
\def\Ent{\mathop{\rm Ent}\nolimits}
\def\exp{\mathop{\rm exp}\nolimits}
\def\dim{\mathop{\rm dim}\nolimits}
\def\argch{\mathop{\rm Argcosh}\nolimits}
\def\ch{\mathop{\rm cosh}\nolimits}
\def\sh{\mathop{\rm sinh}\nolimits}
\def\ci{{\cal C}^\infty}
\def\limti{\dis\lim_{t\to +\infty}}
\def\pa{\partial}
\def\GB{Gau\ss-Bonnet }
\def\hyp{\mathop{\rm hyp}\nolimits}
\def\dvol{\mathop{\rm dvol}\nolimits}
\def\Ker{\mathop{\rm Ker}\nolimits}
\def\ric{\mathop{\rm Ricci}\nolimits}
\def\cX{\ \chi (X)}
\def\kg{\kappa (g)}
\def\can{\mathop{\rm can}\nolimits}
\def\vol{\mathop{\rm vol}\nolimits}
\def\det{\mathop{\hbox{\rm d\'et}}\nolimits}
\def\deg{\mathop{\hbox{\rm deg}}\nolimits}
\def\reg{\mathop{\rm reg}\nolimits}
\def\minvol{\mathop{\rm minvol}\nolimits}
\def\id{\mathop{\rm id}\nolimits}
\def\divv{\mathop{\rm div}\nolimits}
\def\Id{\mathop{\rm Id}\nolimits}
\def\Jac{\mathop{\rm Jac}\nolimits}
\def\conf{\mathop{\rm conf}\nolimits}
\def\Isom{\mathop{\rm Isom}\nolimits}
\def\barr{\mathop{\rm bar}\nolimits}
\def\tr{\mathop{\rm trace}\nolimits}
\def\PSO{\mathop{\rm PSO}\nolimits}
\def\PO{\mathop{\rm PO}\nolimits}
\def\top{\mathop{\rm top}\nolimits}
\def\bs{\overline{s}}
\def\bg{\overline{g}}
\def\bF{\overline{F}}
\def\btheta{\overline{\theta}}

\def\ol{\overline}
\def\vn{\vec\nabla}
\def\hn{\hhf^n}
\def\rn{\rrf^n}
\newcommand\zzf{\mathbf{Z}}
\providecommand{\cad}{c'est-\`a-dire }%
\newcommand\bc{\mathcal{B}}
\newcommand\vc{\mathcal{V}}
\newcommand\scc{\mathcal{S}}

\newcommand\mc{\mathcal{M}}
\newcommand{\ld}{, \ldots,}
\newcommand{\prodd}{\mathop{\prod}\limits}
\newcommand{\summ}{ \mathop{\sum}\limits}
\newcommand{\opp}{ \mathop{\bigoplus}\limits}
\newcommand{\cupp}{ \mathop{\bigcup}\limits}

\def\sumnj{\summ^n_{j=1}}
\def\ixg{\Isom(\wt X, \tilde g_0)}
\providecommand{\cf}{{\upshape cf. }}%

\providecommand{\resp}{{resp. }}%
\newcommand{\wt}{\widetilde}
\newcommand{\wh}{\widehat}

\providecommand{\la}{\longrightarrow}
\newcommand\hhf{\mathbf{H}}
\newcommand\rrf{\mathbf{R}}
\newcommand\hg{\mathfrak{h}}
\newcommand\ggg{\mathfrak{g}}
\newcommand{\dis}{\displaystyle}

\title{Uniform Growth of Groups Acting on Cartan-Hadamard Spaces}
\author{G. Besson, G. Courtois et  S. Gallot}

\maketitle
\centerline{\Large\textbf{Preliminary version}}

\section{Introduction}
In this paper we investigate the growth of finitely generated groups.
Given a group $\Gamma$ generated by a finite set $S$, the word length 
$l_S(\gamma)$ of an element $\gamma \in \Gamma$ is the smallest integer 
$m$ such that there exist elements $\sigma _1,\dots, \sigma _m$ in $S\cup S^{-1}$
with $\gamma = \sigma _1\dots\sigma _m$. The entropy of $\Gamma$ with respect to
the generating set $S$ is defined by 
\begin{equation}
\Ent_S (\Gamma) = \lim_{m\to \infty} \frac{1}{m}{(\log|\{\gamma \in \Gamma\,/\, l_S(\gamma)\leq m\}|)}\,.
\end{equation}
If  $\Ent_S (\Gamma) > 0 $ for some generating set $S$, it is true for all (finite) generating set and the group is said to have exponential growth. We now define the entropy of $\Gamma$ 
\begin{equation}
\Ent \Gamma = \inf _S \{ \Ent _S(\Gamma)\, /\, S\, \rm finite\, generating\, set\, of\, \Gamma\}\,.
\end{equation}
We say that $\Gamma$ has uniform exponential growth if $\Ent \Gamma  > 0$. In \cite{gr1}, remarque 
5.12, M. Gromov raised the question whether exponential growth always implies uniform exponential growth. The answer is negative, indeed,
in \cite{wi} J.S. Wilson gave examples of finitely generated groups of exponential growth and non uniform
exponential growth. Nevertheless,  exponential growth implies uniform
exponential growth for hyperbolic groups \cite{kou}, geometrically finite groups 
of isometries of Hadamard manifolds with pinched negative curvature \cite{a-n}, solvable groups \cite{os} and linear groups \cite{e-m-o}, \cite{Br-Ge}, \cite{Br}.
For further references see the exposition paper \cite{ha}.

We suppose that $(X,g)$ is a $n$-dimensional 
Cartan Hadamard manifold of pinched sectional curvature $-a^2 \leq K \leq -1$.
Our main result is the 
\begin{theo}
There exists a positive constant $C(n,a)$ such that for any finitely generated discrete group $\Gamma$ of isometries
of $(X,g)$, then either $\Gamma$ is virtually nilpotent or $\Ent(\Gamma) \geq C(n,a)$.
\end{theo}

\begin{remark}
The difficulty is here to show that one can choose the constant $C(n,a)$ not depending on the group $\Gamma$.
In the linear setting, E. Breuillard obtained the same kind of uniformity proving 
the existence of a positive constant $C(n)$ such that for any finitely generated group $\Gamma$
of $GL(n,K)$, $K$ any field, then either $\Gamma$ is virtually solvable or
$\Ent(\Gamma) \geq C(n)$.
\end{remark} 
The classical technique is to prove that "not too far" from any finite generating system one can exhibit a free group (in two generators). In this paper we do prove this in one of the cases under consideration, using the famous ping-pong lemma, however in the second case we use a different approach using natural Lipschitz maps from the Cayley graph into $X$. This is the new idea which is described in the following.

In a private communication M.~Kapovich mentioned to us a different proof in the case when $\Gamma$ acts without any elliptic element. One important issue in our proof is that we do not have this restriction, elliptic elements are permitted.

In the forthcoming paper \cite{BCG} we shall use this result to prove a Margulis lemma without curvature; indeed, we shall replace the curvature assumptions by hypothesis on the growth of the fundamental group. 
\section{Preliminaries}
Let $(X,g)$ be a n-dimensional Cartan-Hadamard manifold with sectional curvature 
$-a^2\leq K_g\leq -1$. Let us recall a few well-known facts about isometries. 
If $\gamma$ is an isometry of $(X,g)$, the displacement of $\gamma$ is defined by 
$l(\gamma )=\inf_{x\in X} \rho (x, \gamma x)$, where $\rho$ is the distance associated to the metric $g$ on $X$. 
We then have (see \cite{Ebe} p. 31):
\begin{enumerate}
\item The isometry $\gamma$ is called hyperbolic (or axial) if $l(\gamma )>0$, in which case there exists a 
geodesic $a_\gamma$, called the axis of $\gamma$, such that, for any $x\in a_\gamma$, 
$\rho (x, \gamma x)=l(\gamma)$.
\item The isometry is $\gamma$ is called parabolic if $l(\gamma )=0$ and $l(\gamma )$ is not achieved on $X$, 
in which case there exists a unique point  $\theta$ on the geometric boundary $\partial X$ of $X$ such that 
$\gamma \theta=\theta$.
\item The isometry $\gamma$ is called elliptic if $l(\gamma )=0$ and $l(\gamma )$ is achieved on $X$, 
in which case there exists a non empty convex subset $F_\gamma$ of $X$ such that, for any $x\in F_\gamma$, 
$\gamma x=x$.
\end{enumerate}

The following result, due to G.~Margulis, describes the structure of discrete subgroups of isometries generated 
by elements with small displacement.

\begin{theo}[G.~Margulis, \cite{Bur-Zal}]\label{margulis-lemma}
There exists a constant $\mu (n,a)>0$ such that if $\Gamma$ is a discrete subgroup of the isometry group 
of $(X,g)$, the subgroup $\Gamma_\mu$ of $\Gamma$ generated by,
$$S_\mu=\{ \gamma\in \Gamma / \rho (x, \gamma x)\leq \mu (n,a)\}\,,$$
is virtually nilpotent.
\end{theo}

Given a set of isometries $S=\{\sigma_1, \dots, \sigma_p \}$ of $(X,g)$, we define the ``minimal
displacement'' of $S$ by 
\begin{definition}\label{0}
$L(S)= \inf_{x\in X}\max_{i=1,\dots p} \rho(x,\sigma_i x)$
\end{definition} 
When $\Gamma$ is a finitely generated discrete subgroup of the isometry group of $(X,g)$,
the above theorem \ref{margulis-lemma} has the following
\begin{corollary}\label{margulis-lemma2}
There exists a constant $\mu (n,a)>0$ such that if $\Gamma$ is a finitely generated not virtually nilpotent
discrete subgroup of isometry of $(X,g)$ 
and $S=\{\sigma_1, \dots, \sigma_p \}$ a finite generating set of $\Gamma$, then
$$L(S)\geq \mu (n,a)\}\,.$$
\end{corollary}

In the following lemma we describe the structure of virtually nilpotent discrete subgroups of isometries of 
$(X,g)$. Here by discrete we mean that the orbits are discrete sets in $(X,g)$.
\begin{lemma}\label{virtually-nilpotent}
Let $G$ be a discrete virtually nilpotent group of isometries of $(X,g)$.
\begin{enumerate}
\item[a)] If $G$ contains an hyperbolic element $\gamma$, then $G$ preserves the axis of $\gamma$.
\item[b)] If $G$ contains a parabolic element $\gamma$ with fixed point $\theta\in \partial X$, then $G$ 
fixes the point $\theta$.
\item[c)] If all elements of $G$ are elliptic, then $G$ is finite.
\end{enumerate}
\end{lemma}
\begin{proof}
a) Let $\gamma\in G$ be an hyperbolic element and $\theta , \zeta \in \partial X$, the end points of the axis 
$a_\gamma$ of $\gamma$. We claim that for any $\gamma' \in G$, then 
$\gamma' (\{ \theta, \zeta\})=\{ \theta, \zeta\}$ or $\gamma' (\{ \theta, \zeta\})\cap\{ \theta, \zeta\}=\emptyset$. 
Indeed assume that $\gamma' (\{ \theta, \zeta\})\cap\{ \theta, \zeta\}=\{\theta \}$. The isometry 
$\gamma'\gamma \gamma'^{-1}$ is hyperbolic with axis
$a_{\gamma'\gamma \gamma'^{-1}}=\gamma' a_\gamma$ equal to the geodesic joining $\theta$ and $\zeta'\ne \zeta$, 
where   $\gamma' (\{ \theta, \zeta\})=\{ \theta, \zeta'\}$. We may assume that $\theta$ is the attractive fixed 
point of $\gamma$ and $\gamma'\gamma \gamma'^{-1}$ (replacing them by their inverse if necessary). Let 
$x\in a_\gamma$, then $(\gamma'\gamma \gamma'^{-1})^{-N}\gamma^Nx$ is a sequence of 
pairwise distinct points which converges to a point on the axis $a_{\gamma'\gamma \gamma'^{-1}}$ of 
$\gamma'\gamma \gamma'^{-1}$. This contradicts the discreteness of $G$, proving thus the claim. 

Now, if there exist $\gamma'\in G$ such that $\gamma' (\{ \theta, \zeta\})\cap\{ \theta, \zeta\}=\emptyset$, the two hyperbolic isometries $\gamma$ and  $\gamma'\gamma \gamma'^{-1}$ would then have disjoint axis and therefore $G$ would contain a free subgroup by a classical ping-pong argument. This would contradict the fact that $G$ is virtually nilpotent. Consequently, for any $\gamma'\in G$,  $\gamma' (\{ \theta, \zeta\})=\{ \theta, \zeta\}$, which shows that $G$ preserves the geodesic joining $\theta$ and $\zeta$.

b) Let $\gamma\in G$ be a parabolic element, and $\theta\in \partial X$ its fixed point. If there exist $\gamma'\in G$ such that $\gamma'\theta\ne \theta$, then $\gamma$ and $\gamma'\gamma \gamma'^{-1}$ would be to parabolic elements in $G$ with distinct fixed point $\theta$ and $\gamma'\theta$ respectively. By a ping-pong argument, $G$ would then contain a free subgroup, which contradicts the fact that $G$ is virtually nilpotent. Thus $G$ fixes $\theta\in\partial X$.

c) Let us now assume that all elements in $G$ are elliptic. Let $N\subset G$ be a nilpotent subgroup of $G$ with finite index. If $N=\{ e\}$, then $G$ is finite. We thus assume that $N\ne\{ e\}$, the center $Z(N)$ of $N$ is then not trivial. For $g_1\in Z(N)\setminus \{ e\}$ let us denote $F_{g_1}\subset X$ the set of fixed points of $g_1$. Let $x_1\in F_{g_1}$; by commutation of $g_1$ and $\exp_{x_1}$, we have $F_{g_1}=\exp_{x_1}(E_1)$, where $E_1$ is the eigenspace of $d_{x_1}g_1$ corresponding to the eigenvalue $+1$. This shows that $F_{g_1}$ is a totally geodesic submanifold of $X$ satisfying $\dim (F_{g_1})<\dim (X)$, since $g_1\ne e$. As every $\gamma\in N$ commutes with $g_1$, it satisfies $\gamma (F_{g_1})=F_{g_1}$.

Let $N_1$ be the subgroup of $\Isom (F_{g_1})$ obtained by restriction to $F_{g_1}$ of the elements of $N$; it is clearly nilpotent as the image of a nilpotent group. For $\gamma \in N$, the projection on $F_{g_1}$ of any fixed point of $\gamma$ is again a fixed point of $\gamma$; consequently, the elements of $N_1$ are elliptic elements of $\Isom (F_{g_1})$.

If $N_1=\{ e\}$, then $F_{g_1}$ is pointwise fixed by $N$, therefore $N$ is finite (the group is discrete and all elements have a common fixed point).

If $N_1\ne \{ e\}$, we may iterate the process. Indeed, let us suppose that we have constructed the totally geodesic submanifold $F_{g_i}$, we then construct $N_i$ as the set of restrictions of elements of $N$ to $F_{g_i}$, and, either $N_i=\{ e\}$ in which case $N$ is finite, or $N_i$ is not trivial and, choosing $g_{i+1}\in Z(N_i)\setminus \{ e\}$, we construct the totally geodesic submanifold $F_{g_{i+1}}\in F_{g_i}$ such that $\dim (F_{g_{i+1}})< \dim (F_{g_i})$. This process stops for some $i_0\leq n$ and then $N_{i_0}=\{ e\}$ and $F_{g_{i_0}}$ is pointwise fixed by $N$ and not empty.
\end{proof}

\begin{lemma}\label{fixed-point}
Let $\Gamma$ be a finitely generated discrete group of isometries of $(X,g)$. 

(i) If there exist a point $\theta\in \partial X$ fixed by $\Gamma$, then $\Gamma$ is virtually nilpotent.

(ii) If $\Gamma$ preserves a geodesic in $X$, then $\gamma$ is virtually cyclic.
\end{lemma}

\begin{proof} Proof of (i).
There are three cases: 1) there is an hyperbolic element in $\Gamma$, 2) there is no hyperbolic element, but there is a parabolic element in $\Gamma$ and 3) all elements in $\Gamma$ are elliptic.

1) Let $\gamma$ be a hyperbolic element in $\Gamma$, and $a_\gamma$ its axis. One of the endpoints of $a_\gamma$ is $\theta$. As $\Gamma$ is discrete, it follows from the argument in the proof of lemma \ref{virtually-nilpotent} a) that for any $\gamma'\in \Gamma$, $\gamma' (\{ \theta, \zeta\})=\{ \theta, \zeta\}$ or $\gamma' (\{ \theta, \zeta\})\cap\{ \theta, \zeta\}=\emptyset$, where $\zeta$ is the other  endpoint of $a_\gamma$. Therefore, 
$\gamma' (\{ \theta, \zeta\})=\{ \theta, \zeta\}$ and $\gamma' (\theta )=\theta$ and $\gamma' (\zeta)= \zeta$. The group $\Gamma$ preserves $a_\gamma$. Let us note that $\Gamma$ does not contain any parabolic element, since such an element would fix $\theta$ and therefore also $\zeta$ which is impossible. The elements in $\Gamma$ are thus either hyperbolic or elliptic.

Now, the projection on $a_\gamma$ being distance decreasing, any element $\gamma'\in \Gamma$ achieves its displacement $l(\gamma')$ on the axis $a_\gamma$, and $\gamma'$ is elliptic (resp. hyperbolic) iff $l(\gamma')=0$ (resp. $l(\gamma')\ne 0$). Moreover, since $\gamma'(\theta )=\theta$, any elliptic element fixes pointwise the axis $a_\gamma$. The restriction to the axis $a_\gamma$ is thus a morphism from $\Gamma$ into the group of translations of the axis, whose kernel is the set of elliptic elements, which fix all points of $a_\gamma$ and hence is finite. The group $\Gamma$ is then virtually abelian.

2) In this case the elements of $\Gamma$ are either elliptic or parabolic with fixed point $\theta$. 
In particular, every element of $\Gamma$ preserves each horospheres centred at $\theta$. Indeed, 
this is clear for parabolic elements.
Any elliptic element $\gamma'$ fixes some point $x\in X$, and hence the whole geodesic $c$ 
joining $x$ to $\theta$; let $H$ be any horosphere centred at $\theta$ and $y$ be its 
intersection with $c$, then $\gamma'$ maps $H$ onto the horosphere centred at $\gamma'(\theta )=\theta$ 
containing $\gamma'(y)=y$. This shows that $\gamma'(H)=H$. 

Let $S=\{ \sigma_1, \dots, \sigma_p\}$ be a generating set of $\Gamma$, by the above discussion, 
$\inf_{x\in X}\max_{i\in \{1,\cdots, p\}}\rho (x, \gamma x)=0$. In fact, for any geodesic $c$ such that 
$c(+\infty )=\theta$, let $H_t$ be the horosphere centred at $\theta$ and containing $c(t)$. 
The orthogonal projection from $H_t$ to $H_{t+t'}$ contracts distances, we then get that 
$\rho (c(t), \gamma'(c(t))$ decreases to zero when $t$ goes to infinity, for any 
$\gamma'\in \Gamma$. The group $\Gamma$ is then virtually nilpotent by theorem \ref{margulis-lemma}. 

3) If $\Gamma$ only contains elliptic elements, then for any finite generating set 
$S=\{ \sigma_1, \dots, \sigma_p\}$,  $\inf_{x\in X}\max_{i\in \{1,\cdots, p\}}\rho (x, \sigma_i x)=0$, 
because each $\sigma_i$ preserves each horosphere centred at $\theta$, by the above argument. 
The group $\Gamma$ is again virtually nilpotent.

Proof of (ii).
A subgroup of index two of $\Gamma$ fixes each endpoint of the globally preserved geodesic.
Then we conclude as in the case 1 of (i).
\end{proof}

For any two isometries $\gamma\,,\gamma'$ acting on $(X,g)$ we define,
$$L(\gamma, \gamma')=\inf_{x\in X}\max \{ \rho (x, \gamma x), \rho (x, \gamma' x)\}\,.$$
We now prove the,
\begin{proposition}\label{not-nilpotent}
Let $\Gamma$ be a finitely generated discrete subgroup of $\Isom (X,g)$, where $(X,g)$ is a 
Cartan-Hadamard manifold of sectional curvature $-a^2\leq K_g\leq -1$. Let  
$S=\{ \sigma_1, \dots, \sigma_p\}$ be a finite generating set of $\Gamma$. 
If $\Gamma$ is not virtually nilpotent, we have 
\begin{enumerate}
\item[i)] either there exist $\sigma_i,\sigma_j\in S$ such that the subgroup $<\sigma_i, \sigma_j >$ 
generated by these two elements is not virtually nilpotent and  $L(\sigma_i, \sigma_j )\geq \mu (n,a)$,
\item[ii)] or all $\sigma_i$ in $S$ are elliptic and for all $\sigma_i\ne \sigma_j\in S$, either 
$<\sigma_i, \sigma_j >$ fixes some point in $X$ and is finite, or it fixes a point $\theta\in \partial X$,
\item[iii)] or there exist $\sigma_i, \sigma_j, \sigma_k\in S$ such that 
$L(\sigma_i\sigma_j, \sigma_k )\geq \mu (n,a)$ and the group $<\sigma_i\sigma_j, \sigma_k>$ is not 
virtually nilpotent.
\end{enumerate}
\end{proposition}
\begin{proof}
There are again three cases: a) there is a hyperbolic element in $S$, say $\sigma_1$; b) there is no 
hyperbolic element and there is a parabolic element in $S$, say $\sigma_1$; c) all $\sigma_i$'s in $S$ are elliptic.

a) Let us assume that $\sigma_1$ is hyperbolic. Let us consider all pairs $(\sigma_1, \sigma_i)$ with 
$i=2,\dots , p$, and let us assume that $L(\sigma_1, \sigma_i )<\mu (n,a)$ for $i=2,\dots , p$. 
The groups $<\sigma_1, \sigma_i>$ are then virtually nilpotent. By lemma \ref{virtually-nilpotent} a), every $\sigma_i$ preserves the axis $a_{\sigma_1}$ of $\sigma_1$, hence $\Gamma$ preserves $a_{\sigma_1}$ and is virtually nilpotent contradicting the assumption. Then there exist $\sigma_i\in S$ such that $L(\sigma_1, \sigma_i)\geq \mu (n, a)$ and $<\sigma_1, \sigma_i>$ is non virtually nilpotent.

b) Assume that $\sigma_1$ is parabolic with fixed point $\theta \in \partial X$. Let us consider all pairs 
$(\sigma_1, \sigma_i)$, $i=2,\dots , p$, and assume that $<\sigma_1, \sigma_i>$ is virtually nilpotent 
(or that $L(\sigma_1, \sigma_i)<\mu (n,a)$), for all  $i=2,\dots , p$). By lemma \ref{virtually-nilpotent} b), 
$\sigma_i$ fixes the point $\theta \in \partial X$, therefore $\Gamma$ fixes $\theta$ and is virtually nilpotent, 
by lemma \ref{fixed-point}, a contradiction. Consequently, if 
$\sigma_1$ is parabolic, there exist $\sigma_i\neq \sigma_1$ such that $L(\sigma_1, \sigma_i)\geq \mu (n,a)$.

c) Let us assume that all $\sigma_i$'s are elliptic, for $i=2,\dots , p$, and that for all pairs 
$(\sigma_i, \sigma_j)$ the groups $<\sigma_i, \sigma_j>$ are virtually nilpotent (or that 
$L(\sigma_i, \sigma_j)<\mu (n,a)$). Let us denote $G= <\sigma_i, \sigma_j>$. There are again three cases: 
1) there is a hyperbolic element in $G$, 2) there is no hyperbolic element and there is a parabolic element in 
$G$, 3) all elements in $G$ are elliptic.

In the case 1), let $\gamma$ be a hyperbolic element in $G$ with axis $a_\gamma$. 
By lemma \ref{virtually-nilpotent} a), $G$ preserves $a_\gamma$. Since $\sigma_i, \sigma_j$ are elliptic, 
they fix points $x_i$ 
and $x_j$ (respectively) on $a_\gamma$ (recall that the displacement of $\sigma_i$ and $\sigma_j$ 
are achieved on $a_\gamma$ by the distance decreasing property of the projection onto $a_\gamma$). 
If $x_i=x_j$, the $G$ fixes $x_i$ and it is thus finite. Let us now suppose that 
$\sigma _{i}$ and $\sigma _{j}$ do not fix the same point on $a_{\gamma}$, that is $x_i\ne x_j$
and none of the restriction $\tilde \sigma_i$ and $\tilde \sigma_j$
of $\sigma _{i}$ and $\sigma _{j}$ to $a_\gamma$ is the identity. 
In that case, $\tilde\sigma_i$ and $\tilde\sigma_j$ 
are both symmetries around $x_i$ and $x_j$, and then $\sigma_i\sigma_j$ is a hyperbolic element with axis 
$a_\gamma$. Let us then consider $<\sigma_i\sigma_j, \sigma_l>$ 
for $l=1,\dots,p$. Assume that for all $l=1,\dots,p$, $L(\sigma_i\sigma_j,\sigma_l)<\mu (n,a)$, the groups 
$<\sigma_i\sigma_j, \sigma_l>$ are then virtually nilpotent, and by lemma \ref{virtually-nilpotent} a), all 
$\sigma_l$'s preserve $a_\gamma$ and hence $\Gamma$ preserves  $a_\gamma$ and is thus virtually nilpotent 
which is a contradiction. Therefore, there exist $\sigma_k\in S$ such that 
$L(\sigma_i\sigma_j,\sigma_k)\geq\mu (n,a)$ and that $<\sigma_i\sigma_j,\sigma_k>$ is not virtually nilpotent.

In the case 2), let $\gamma\in G$ be a parabolic element with fixed point $\theta\in \partial X$. 
By lemma \ref{virtually-nilpotent} b), $G$ fixes $\theta$.

In the case 3), all elements in $G$ are elliptic and by lemma \ref{virtually-nilpotent} c), $G$ is finite.
This ends the proof of the proposition.
\end{proof}

\section{Algebraic length and $\eta$-straight isometries}

Let $\Gamma$ be a finitely generated discrete group of isometries of $(X,g)$ and 
$S=\{ \sigma_1, \dots, \sigma_p\}$ be 
a finite generating set of $\Gamma$. 

Let us denote $l_S$ and $d_S$ the length and distance on the Cayley graph associated to $S$. 
Let $x_0$ be a point in $X$ 
and define $L=\max_{i\in \{1,\dots, p\}}\rho (x_0,\sigma_ix_0)$.

For any $\gamma\in\Gamma$ it follows from the triangle inequality that
\begin{equation}
\rho (x_0,\gamma x_0)\leq  l_S(\gamma)L\,.
\end{equation} 
Let $\eta$ be a positive number such that $0<\eta<L$.
\begin{definition}
An isometry $\gamma$ of $\Gamma$ is said to be $(L,\eta )$-straight if $\rho(x_0, \gamma x_0)\geq (L-\eta )l_S(\gamma )$.
\end{definition}
\begin{remark}
Notice that the above definition depends on the choice of $x_0$ and of a generating set $S$.
\end{remark}
When $\Gamma$ is a finitely generated discrete group, for any finite generating set 
$S=\{ \sigma_1, \dots, \sigma_p\}$ 
we define,
$$L(S)=\inf_{x\in X}\max_{i\in\{ 1,\dots,p\}}\rho(x,\sigma_ix)\,.$$
When $\Gamma$ is not virtually nilpotent, by theorem \ref{margulis-lemma}, for any finite generating set $S$, 
$L=L(S)\geq \mu (n,a)>0$, where $\mu (n,a)$ is the Margulis constant. We then have,
\begin{lemma}\label{x0}
Let $\Gamma$ be a finitely generated non virtually nilpotent discrete group of isometries of $(X,g)$. 
For any finite 
generating set $S=\{ \sigma_1, \dots, \sigma_p\}$ of $\Gamma$, there exist $x_0\in X$ such that,
$$L(S)=\inf_{x\in X}\max_{i\in\{ 1,\dots,p\}}\rho(x,\sigma_ix)=\max_{i\in\{ 1,\dots,p\}}\rho(x_0,\sigma_ix_0)\,.$$
\end{lemma}
\begin{proof}
Let us assume that the infimum in the definition of $L(S)$ is not achieved in $X$, then there exist a sequence of 
points $x_k\in X$, which satisfies $\max_{i\in\{ 1,\dots,p\}}\rho(x_k,\sigma_ix_k)\to L(S)$ when $k\to \infty$, 
and $x_k$ converges to a point, say $\theta$, in $\partial X$. For $k$ large enough and $i\in\{ 1,\dots,p\}$, we 
then have $\rho (x_k,\sigma_ix_k)\leq L+1$ and hence $\sigma_i\theta=\theta$ for all $i$. This shows that $\Gamma$ 
fixes $\theta$ and is thus virtually nilpotent by lemma \ref{fixed-point}, which contradicts the hypothesis.
\end{proof}

In the sequel of this section, we shall show that if $G$ is a finitely generated discrete group of isometries 
of $(X,g)$, for any finite generating set  $S=\{ \sigma_1, \dots, \sigma_p\}$ of $G$ such that each $\sigma_i$ has 
a displacement $l(\sigma_i)$ small compared to $L(S)$, then there exist many non-$(L(S),\eta)$-straight elements 
in $G$ for a constant $\eta$ to be defined.

We need the following geometric lemmas.
\begin{lemma}\label{three-points}
Let $(x_1,x_2,x_3)$ be a geodesic triangle in $(X,g)$, where $(X,g)$ is a Cartan-Hadamard manifold with 
$K_g\leq -1$. 
Let $x_2'$ be the point in the segment $[x_1,x_3]$ dividing it in two segments of length 
proportional to $L_1:=\rho (x_1,x_2)$ and   $L_2:=\rho (x_2,x_3)$. We have,
$$\rho (x_2',x_2)\leq \argch\Big [ \exp\Big (\alpha\big (\rho(x_1,x_2)
+\rho (x_2,x_3)-\rho (x_1,x_3)\big )\Big )\Big ]\,,$$
where $\alpha={\max (L_1,L_2)\over L_1+L_2}$.
\end{lemma}

\begin{proof}
We consider a comparison geodesic triangle $(y_1,y_2,y_3)$ in the Poincar\'e disk $(\mathbb H^2,d)$ of constant curvature
$-1$ such that $d(y_i,y_j) =\rho(x_i, x_j)$ for all $i, j \in \{1,2,3 \}$. Let $ y'_2$ be the point of the
segment  $[y_1,y_3]$ dividing it in two segments of length 
proportional to $L_1$ and   $L_2$. 
Since $(X,g)$ is a $CAT(-1)$ space we have 
\begin{equation}\label{2}
\rho(x_2,x'_2) \leq d(y_2,y'_2).
\end{equation}
One of the two triangles $(y_1,y'_2,y_2)$,  $(y_3,y'_2,y_2)$ has angle at $y'_2$ greater than or equal to $\pi /2$,
therefore from hyperbolic trigonometry formulae we get the existence of $i\in \{1, 2 \}$ such that
\begin{equation}\label{3}
\cosh L_i \geq \cosh \big[d(y_2,y'_2)\big] \cosh \Big[ \frac{L_i}{(L_1 +L_2)} d(y_1,y_3)\Big]
\end{equation}
Let us denote $\Delta = \rho(x_1,x_2) + \rho(x_2,x_3)- \rho(x_1,x_3)$.
We have 
\begin{equation}
\frac{L_i}{L_1+L_2} d(y_1,y_3) \geq L_i - \alpha \Delta,
\end{equation}
where $\alpha={\max (L_1,L_2)\over L_1+L_2}$.
Therefore from (\ref{2}) and (\ref{3}) we get 
\begin{equation}\label{3bis}
\cosh \big[\rho(x'_2,x_2)\big] \leq \frac{\cosh L_i}{\cosh (L_i - \alpha \Delta)}\,,
\end{equation}
hence
\begin{equation}\label{3ter}
\cosh \big[\rho(x'_2,x_2)\big] \leq e^{\alpha \Delta}.
\end{equation}
\end{proof}

\begin{lemma}\label{square}
Let $(X,g)$ be a Cartan-Hadamard manifold with sectional curvature $K_g\leq -1$. Let $\delta, L$ 
be any positive numbers 
such that $L > \argch (e^\delta)$. Then, for any isometry $\gamma$ of $(X,g)$ such that its displacement 
$l(\gamma )$ 
satisfies $l(\gamma )\leq \delta$, and for any point $x_0\in X$ such that $\rho (x_0, \gamma x_0)\geq L$, we have
$$\rho (x_0,\gamma^2x_0)\leq 2\rho (x_0, \gamma x_0)-\Big(1-\frac{e^\delta}{\cosh L}\Big)^2\,,$$
\end{lemma}

\begin{proof}
Let us consider $\Delta = 2\rho(x_0,\gamma x_0) - \rho (x_0, \gamma^2 x_0)$.
We want to prove that $\Delta \geq \Big(1-{e^{\delta} \over \cosh L}\Big)^2$.
By assumption there is a point $y\in X$ such that $\rho(y,\gamma y) \leq \delta$.
Let us write $L_1 =: \rho(x_0, \gamma y)$, $L_2 =: \rho(\gamma ^2 x_0, \gamma y)$ and 
$L' =: \rho(x_0, y)$. By the triangle inequality we have for $i=1 , 2$
\begin{equation}\label{onze}
L' - \delta \leq L_i \leq L'+ \delta\,.
\end{equation} 
Let us associate to the triangle $(x_0, \gamma y, \gamma ^2 x_0)$ 
the comparison triangle $(z_1, z_2, z_3)$ in the hyperbolic plane $(\mathbb H^2,d)$ such that 
$d(z_1, z_2) = L_1$, $d(z_2, z_3) = L_2$ and $d(z_1, z_3)= \rho( x_0, \gamma ^2 x_0)$. 
Let $x$ [resp. $z$] be the middle point of the segment $(x_0, \gamma ^2 x_0)$ [resp. $(z_1, z_3)]$.
One of the two triangles $(z_2,z,z_1)$ or $(z_2,z,z_3)$ has angle at $z$ greater than or equal to
$\pi /2$. Let us assume without restriction that this triangle is $(z_2,z,z_1)$, then 
the hyperbolic trigonometric formulas give
$$
\cosh L_1 \geq \cosh \Big[d(z_2,z) \Big] \cosh \Big[ {1\over 2} d(z_1,z_3)\Big] 
$$
therefore from (\ref{onze}) we get 
$$
\cosh (L'+\delta) \geq \cosh \Big[d(z_2,z) \Big] \cosh \Big[ {1\over 2} d(z_1,z_3)\Big] 
$$
and since $(X,g)$ is a $CAT(-1)$ space we have $\rho(x,\gamma y) \leq d(z_2,z)$, thus we obtain
\begin{equation}\label{douze}
\cosh (L'+\delta) \geq \cosh \Big[\rho(x,\gamma y) \Big] \cosh \Big[ {1\over 2} \rho(x_0, \gamma ^2 x_0)\Big]\,. 
\end{equation}
Let us write $L_0 = \rho(x_0,\gamma x_0)$.
By the triangle inequality we have  
$$
\rho(x,\gamma y) \geq | \rho(\gamma y, \gamma x_0) - \rho(\gamma x_0, x)|\,,
$$ 
therefore, since $ \rho(\gamma y, \gamma x_0)= \rho(y, x_0) = L'$
and ${1\over 2}\rho(\gamma ^2 x_0, x_0) = L_0 - {\Delta \over 2}$, we get from \ref{douze}
\begin{equation}\label{treize}
\cosh (L'+\delta) \geq \cosh \Big(L' - \rho(\gamma x_0, x) \Big)\cosh \Big( L_0 - {\Delta \over 2}\Big)\,.
\end{equation}
We get from (\ref{treize}),
$$
(\cosh \delta +\sinh \delta)\cosh L' \geq
\Big(\cosh \big[\rho(\gamma x_0, x)\big]-\sinh \big[\rho(\gamma x_0, x)\big]\Big)
\big(\cosh L'\big)\cosh \big(L_0 - {\Delta \over 2}\big)
$$
hence
\begin{equation}\label{quatorze}
e^\delta \geq \Big(\cosh \big[\rho(\gamma x_0, x)\big]-\sinh \big[\rho(\gamma x_0, x)\big]\Big)
\cosh \big(L_0 - {\Delta \over 2}\big)\,.
\end{equation}
Now applying the inequality \ref{3bis} in the proof of lemma
\ref{three-points} we have 
$$
\cosh \big[\rho(\gamma x_0, x)\big] \leq \frac{\cosh L_0}{\cosh \big(L_0 -{\Delta\over 2}\Big)}\,,
$$
and since $\cosh r -\sinh r = e^{-r}$ is a decreasing function of $r$ we get from (\ref{quatorze})
\begin{equation}\label{quinze}
e^\delta \geq \cosh L_0 - \Big(\cosh ^2 L_0 
-\cosh ^2 \big(L_0 - {\Delta \over 2}\big)\Big)^{1\over 2}\,.
\end{equation}
But we can check that $\cosh ^2 L_0 - \cosh ^2 \big(L_0 - {\Delta \over 2}\big)\Big) \leq \Delta \, \cosh ^2L_0$
so we get from (\ref{quinze})
$$
e^\delta \geq \cosh (L_0) \big(1 - \Delta ^{1\over 2}\big)
$$
and therefore
$$
\Delta
\geq \Big(1-\frac{e^\delta}{\cosh L_0}\Big)^2,
$$
when $e^\delta < \cosh L$.
The lemma now follows whenever $L_0 \geq L$.

\end{proof}

\begin{lemma}\label{four-points}
Let $(X,g)$ be a Cartan-Hadamard manifold with sectional curvature $K_g\leq -1$. 
Let us consider four points $y_0,y_1,y_2,y_3$ such that 
$$
\rho (y_0,y_1)+\rho (y_1,y_2)-\rho (y_0,y_2)\leq \eta_1
$$ and 
$$\rho (y_1,y_2)+\rho (y_2,y_3)-\rho (y_1,y_3)\leq \eta_2$$
then
$$\rho (y_0,y_1)+\rho (y_1,y_2)+\rho (y_2,y_3)-\rho (y_0,y_3)\leq 
\Big ( 1+{\rho (y_2,y_3)\over \rho (y_1,y_2)}\Big )\Big (\eta_1+\argch e^{\eta_2}\Big )\,.$$
\end{lemma}
\begin{proof}
For $i=1,2, 3$  let us write $L_i = \rho(y_{i-1}, y_i)$. Let $y'_2$ be the point on the segment 
$(y_1,y_3)$ dividing it in two segments of length proportional to $L_2$ and $L_3$.
By lemma \ref{three-points} we have
\begin{equation}\label{seize}
\rho(y_2,y'_2) \leq \argch\big(e^{\eta_2}\big)\,.
\end{equation}
Since $\rho (y_0,y_1)+\rho (y_1,y_2)-\rho (y_0,y_2)\leq \eta_1$ by assumption
we get from (\ref{seize}) and the triangle inequality
\begin{equation}\label{17}
\rho(y_0,y'_2)\geq \rho(y_0,y_2)-\rho(y_2,y'_2)\geq \rho (y_0,y_1)+\rho (y_1,y_2)
-\big[\eta_1+ \argch\big(e^{\eta_2}\big)\big]
\end{equation}
On the other hand by convexity of the distance function on $(X,g)$ we get
\begin{equation}\label{18}
\rho(y_0,y'_2) \leq  \frac{L_3}{L_2 +L_3}\rho (y_0,y_1)+ \frac{L_2}{L_2 +L_3} \rho (y_0,y_3)
\end{equation}
The inequalities (\ref{17}) and (\ref{18}) give
$$
\rho(y_0,y_3)\geq \rho (y_0,y_1) +L_2 +L_3 - 
\Big(\frac{L_2 +L_3}{L_2}\Big)\Big(\eta_1+ \argch \big(e^{\eta_2}\big)\Big)
$$
and the lemma follows.
\end{proof}

\begin{lemma}\label{products}
Let $L$ and $\eta$ be two positive numbers such that,
$$\eta<\min \Big ({L\over 4}, {1\over 2}\log \Big [{1\over 2}\big (\ch ({L\over 2})
+{1\over \ch ({L\over 2})}\big )\Big ]\Big )\,.$$
Let $(X,g)$ be a Cartan-Hadamard manifold with sectional curvature $K_g\leq -1$. 
We consider two elliptic isometries $\gamma_1,\gamma_2$ of $(X,g)$ with a common fixed point 
$y\in X\cup \partial X$. 
If we assume that $L-\eta\leq \rho (x_0, \gamma_1 x_0)\leq L$ and that 
$L-\eta\leq \rho (x_0, \gamma_2 x_0)\leq L$, then 
$$\rho (x_0, \gamma_1\gamma_2x_0)<2(L-\eta )\,.$$
\end{lemma}
\begin{proof}
We first claim that in both cases, $y\in X$ and $y\in \partial X$, there exist some sequence 
$(u_k)_{k\in\textbf{N}}$ of 
points in $X$ converging to $y$ 
such that $\rho (u_k, \gamma_1\gamma_2x_0)=\rho (u_k,x_0)=l_k$ and that the quantity 
$\epsilon_k=|\rho (u_k, \gamma_1x_0)-l_k|$ goes to zero when $k$ goes to $+\infty$; in fact, when 
$\gamma_1$ and $\gamma_2$ fix some point $y\in X$ we may choose $u_k=y$ for every $k$. If $\gamma_1$ and 
$\gamma_2$ fix $y\in \partial X$, they also preserve each horosphere centred at $y$ (see the proof of 
lemma \ref{fixed-point} 2)), and thus  $x_0$, $\gamma_1x_0$ and $\gamma_1\gamma_2x_0$ lie on the same 
horosphere centred at $y$. Approximating this horosphere by a sequence $(S_k)_{k\in\textbf{N}}$ of 
spheres passing through $x_0$ and $\gamma_1\gamma_2x_0$ and denoting $u_k$ the centre of $S_k$, we 
get that $\rho (u_k,\gamma_1x_0)-\rho (O,u_k)$ and $\rho (u_k,x_0)-\rho (O, u_k)$ simultaneously go to 
$B(\gamma_1x_0, y)=B(x_0,y)$ (where $O$ is some fixed origin in $X$ and $B$ the Busemann function 
normalised at $O$). This proves the claim.

Consider the triangle $(u_k,v,w)=(u_k, x_0, \gamma_1\gamma_2x_0)$ and $z$ 
the point of the geodesic segment 
$[v,w]$ which divides it in two segments of length proportional
to $L_1 =: \rho(v,\gamma_1x_0)$ and  $L_2 =: \rho(w,\gamma_1x_0)$.
Let us recall that by assumption we have $L-\eta \leq L_i \leq L$.

 We consider the comparison triangle $(\bar u_k, \bar v, \bar w)$ on the two-dimensional 
hyperbolic space $\textbf{H}^2$ such that 
$d(\bar u_k, \bar v)= \rho (u_k, v)=l_k=\rho (u_k,w)=d(\bar u_k,\bar w)$ and 
$d (\bar v, \bar w)=\rho (v,w)$, where $d$ is the hyperbolic distance on $\textbf{H}^2$. 
Let $\bar z$ be the point of the segment $[\bar u, \bar v]$ dividing it in two segments of length
proportional to $L_1$ and $L_2$.
Let us write $L'_1 = \rho(v,z)$ and $L'_2 = \rho (w,z)$. 
We now consider the triangle $(\bar u_k,\bar v,\bar z)$ or 
$(\bar u_k,\bar w,\bar z)$, namely the one which has angle at $\bar z$ larger than
or equal to $\pi /2$.
We can assume without restriction that this triangle is  $(\bar u_k,\bar v,\bar z)$.
The hyperbolic trigonometry formulas then show that the point $\bar z$ satisfies,
$$\ch (l_k) \geq \ch (L'_1)\ch (d(\bar u_k, \bar z))\,.$$
Since $(X,g)$ is a CAT(-1)-space, we get that,
$$\rho (u_k,z)\leq d(\bar u_k, \bar z)\,,$$
and thus that,
\begin{equation}\label{one}
\ch (\rho (u_k,z))\leq {\ch (l_k)\over \ch (L'_1)}\,.
\end{equation} 
On the other hand, the triangle inequality implies that 
$\rho (u_k,z)\geq l_k-\epsilon_k-\rho (\gamma_1x_0,z)$ and thus that 
$$\cosh (\rho (u_k,z))\geq e^{-(\rho (\gamma_1x_0,z)+\epsilon_k)}\ch (l_k)\,.$$
Plugging this in formula \ref{one} and letting $\epsilon_k\to 0$, we get:
\begin{equation}\label{two}
e^{\rho (\gamma_1x_0,z)}\geq \cosh(L'_1)\,.
\end{equation}
On the other hand, by lemma \ref{three-points}, we have 
$$\cosh (\rho (\gamma_1x_0,z))
\leq \exp\Bigg (\max\Big\{\rho (v, \gamma_1x_0), \rho (w, \gamma_1x_0)\Big\}  
\Big (1-{\rho (v,w)\over \rho (v,\gamma_1x_0)+\rho (w, \gamma_1x_0)}\Big )  \Bigg )\,.$$
and hence 
\begin{equation}\label{three}
\cosh (\rho (\gamma_1x_0,z))\leq e^{(L-{\rho (v,w)\over 2})}\,.
\end{equation}
Let us now assume, by contradiction, that 
$$\rho (v,w)=\rho (x_0, \gamma_1\gamma_2x_0) > 2(L-\eta )\,.$$
Plugging this in the inequalities (\ref{two}) and (\ref{three}) we obtain,
using the fact that $x\to x+1/x$ is an increasing function for $x>1$:
\begin{equation}\label{four}
\ch (L'_1)+{1\over \ch (L'_1 )}\leq 2 \ch (\rho (\gamma_1x_0,z))\leq 2e^\eta.
\end{equation}
Now since $\frac{L'_1}{L'_2} = \frac{L_1}{L_2}$, we also obtain 
$$L'_1=(L'_1+L'_2) \big( \frac{L_1}{L_1+L_2} \big) \geq \frac{2(L-\eta)(L-\eta)}{2L} \geq L-2\eta$$
which gives by inequalities (\ref{four})
\begin{equation}
\ch (L -2\eta)+{1\over \ch (L-2\eta)}\leq 2 \ch (\rho (\gamma_1x_0,z))\leq 2e^\eta.
\end{equation}
we then get a contradiction when
$$\eta <
\min \Big ({L\over 4}, {1\over2}\log \Big [{1\over 2}\big (\ch ({L\over 2})+{1\over \ch ({L\over 2})}\big )\Big ]\Big )\,.$$
\end{proof}

Let $\Gamma$ be a finitely generated discrete group of isometries of $(X,g)$ and $S=\{
\sigma_1,\dots, \sigma_p\}$ be a finite generating set. Let us assume that $\Gamma$ is not virtually nilpotent 
and recall that $L(S)=\inf_{x\in X}\max_{i\in\{ 1,\dots,p\}}\rho(x,\sigma_ix) $. By lemma \ref{x0} we have 
$L(S)=\max_{i\in\{ 1,\dots,p\}}\rho(x_0,\sigma_ix_0)$, for some point $x_0\in X$, and by corollary
\ref{margulis-lemma2}, $L(S)\geq \mu (n,a)>0$. 
Let us recall that for $0 \leq \eta \leq L$, an element $\gamma \in \Gamma$ is said 
to be $(L,\eta)$-straight if 
$$
\rho(x_0, \gamma\,x_0) > (L- \eta)l_S(\gamma) \,.
$$
In the following two propositions we give conditions under which
there are many non $(L,\eta)$-straight elements in $\Gamma$.
\begin{proposition}\label{elliptic}
Let $(X,g)$ be a Cartan Hadamard manifold whose sectional curvature satisfies
$-a^2 \leq K \leq -1$ and $\Gamma$ a discrete non virtually nilpotent group of isometries of $(X,g)$
generated by $S=\{ \sigma_1,\dots,\sigma_p\}$.
Let us assume that all $\sigma_i$'s are elliptic and that for all $\sigma_i\ne \sigma_j\in S$, the group 
$<\sigma_i,\sigma_j>$ fixes a point $y\in X$ or $\theta\in \partial X$. Let $\eta$ be a positive number such that 
$$\eta<\min\Bigg ({L\over 4}, {1\over 2}\Big (1-{1\over \ch (L)}\Big )^2, 
{1\over 2}\log\Big [{1\over 2}(\ch {L\over 2}+ {1\over \ch {L\over 2}})\Big ]\Bigg )\,,$$
where $L=L(S)= \inf_{x\in X}\max_{i\in\{ 1,\dots,p\}}\rho(x,\sigma_ix)
=\max_{i\in\{ 1,\dots,p\}}\rho(x_0,\sigma_ix_0)$,
then any 
$\gamma\in \Gamma$ with 
$l_S(\gamma )=2$, \textsl{i.e} 
$\gamma=\sigma_i^2$ or $\gamma=\sigma_i\sigma_j$, is not $(L,\eta /2)$-straight, that is, 
$\rho (x_0,\gamma x_0)\leq 2(L-\eta /2)$\,.
\end{proposition}
\begin{proof}
Consider the case where $\gamma=\sigma_i^2$. If $\sigma_i$ is not 
$(L,\eta)$-straight, we have, by the triangle inequality, $\rho (x_0,\sigma_i^2x_0)\leq 2(L-\eta)$.
If $\sigma_i$ is $(L,\eta)$-straight, we have by lemma \ref{square},
$$\rho (x_0, \sigma_i^2x_0)\leq 2L-\Bigg(1-{1\over \ch L}\Bigg)^2\leq 2(L-\eta )\,.$$
Let us now consider the case where $\gamma=\sigma_i\sigma_j$, for $i\ne j$. If $\sigma_i$ or $\sigma_j$ 
is not $(L,\eta)$-straight, we have, by the triangle inequality,
$$\rho (x_0, \sigma_i\sigma_jx_0)\leq \rho (x_0, \sigma_ix_0)+\rho ( x_0, \sigma_jx_0)\leq L+(L-\eta )\,,$$
therefore, $\rho (x_0, \sigma_i\sigma_jx_0 )\leq 2 (L-\eta/2)$.

If $\sigma_i$ and $\sigma_j$ are $(L,\eta)$-straight, lemma \ref{products} implies that 
$\rho (x_0, \sigma_i\sigma_jx_0)\leq 2(L-\eta )$.
\end{proof}

In the next proposition we will assume that 
all elements $\gamma\in\Gamma$ whose algebraic length is less than or equal to $4$ have 
a displacement smaller than $\delta $ where 
\begin{equation}\label{delta} 
\delta = \log\big[\cosh ({L\over 4})\big]\,,
\end{equation}
and we set 
\begin{equation}\label{eta}
\eta = 10^{-3}\Big(1- \frac{\cosh({L\over4})}{\cosh({L\over2})}\Big)^4\,.
\end{equation}
We will find in that case many non $(L,\eta)$-straight elements. 
\begin{proposition}\label{length-6}
Let $(X,g)$ be a Cartan-Hadamard manifold whose sectional curvature satisfies $\-a^2 \leq K_g\leq -1$
and $G$ a discrete non virtually nilpotent group of isometries of $(X,g)$ generated by 
a set of two isometries
$\Sigma = \{\sigma_1, \sigma _2\}$. Let
$L =\inf_{x\in X}\max\{\rho(x,\sigma_1 x),\rho(x,\sigma_2 x)\}\,,$ and
$\Sigma = \{\sigma_1,\sigma_2\}$.
Let $\eta$ and $\delta$ be the numbers defined in (\ref{eta}) and (\ref{delta}).
We assume that $l(\gamma')< \delta$ for all 
$\gamma' \in G$ such that $l_{\Sigma}(\gamma')< 4$.
Then all elements $\gamma \in G$ such that  $l_{\Sigma}(\gamma)=6$ are not
$(L,\eta)$-straight.
\end{proposition}
We will need the following lemmas.
\begin{lemma}\label{technical-lemma}
Let $\gamma =a \gamma' b \in G$ be such that 
$l_{\Sigma}(\gamma) = l_{\Sigma}(a) +l_{\Sigma}(\gamma')+l_{\Sigma}(b)$.
If $\gamma$ is $(L,\eta)$-straight, then $\gamma'$ is  $(L,C \,\eta)$-straight where
$C= \frac{l_{\Sigma}(\gamma)}{l_{\Sigma}(\gamma')}$.
\end{lemma} 
\begin{proof}
Let us note that by definition of $L=L(\Sigma)$, we have for any $\gamma \in G$
$$
\rho(x_0,\gamma x_0) \leq L\,. l_{\Sigma}(\gamma)\,.
$$
By triangle inequality we have 
$$
\rho(x_0,\gamma x_0) \leq \rho(x_0,a x_0)+ \rho(x_0,\gamma' x_0)+\rho(x_0,b x_0)\,,
$$
hence by assumption on $\gamma$ we get
$$
(L-\eta) \,l_{\Sigma}(a \gamma' b) \leq L\, \big(l_{\Sigma}(a) +l_{\Sigma}(b)\big) +\rho(x_0,\gamma' x_0)
$$
and therefore,
$$
\rho(x_0,\gamma' x_0) \geq L \,l_{\Sigma}(\gamma') - \eta \, l_{\Sigma}(\gamma) 
\geq (L -C\, \eta)l_{\Sigma}(\gamma')\,.
$$
\end{proof}
\begin{lemma}\label{alpha-square-alpha-beta-alpha}
Let $\alpha$, $\beta$ be two elements of $G$ distinct of the neutral element and such that
$l_{\Sigma}(\alpha) \leq 2$,  $l_{\Sigma}(\beta) \leq 2$. Under the assumptions of the proposition
\ref{length-6}, if $\gamma$ is $(L,\eta)$-straight with $l_{\Sigma}(\gamma) = 6$, then any reduced word
representing $\gamma$ does not contain $(i)$ $\alpha ^2$ or $(ii)$ $\alpha \beta \alpha$.
\end{lemma}
Assuming the lemma \ref{alpha-square-alpha-beta-alpha}, the proof of the proposition \ref{length-6}
can be finished as follows: 
\begin{proof}
Let $\gamma \in G$ of length $l_{\Sigma}(\gamma)=6$. Let us write $\gamma$ as a reduced word
in the generators of $\Sigma$, $\gamma = \sigma_{i_1}^{p_1}\,.\, \dots\,. \sigma_{i_k}^{p_k}$,
where $\sigma _{i_j} = \sigma_1$ or  $\sigma _{i_j} = \sigma_2$, $p_j \in \mathbb Z ^{\ast}$, $i_j \neq i_{j+1}$
and $i_j =i_{j+2}$. Arguing by contradiction we assume that $\gamma$ is 
$\eta$-straight. Then, by lemma \ref{alpha-square-alpha-beta-alpha} (i), all $p_j$ are
equal to $+1$ or $-1$ and in particular we have $k=6$. Therefore
$ \gamma =\sigma_{i_1}^{p_1}\,.\,\sigma_{i_2}^{p_2}\,.\,\sigma_{i_3}^{p_3}\,.\, \sigma_{i_4}^{p_4}\,.\,
\sigma_{i_5}^{p_5}\,. \sigma_{i_6}^{p_6}\,.$
By lemma \ref{alpha-square-alpha-beta-alpha} (ii) we also have $p_{j+2} \neq p_j$ hence 
$p_{j+2} = -p_j$ so 
$ \gamma =\sigma_{i_1}^{p_1}\,.\,\sigma_{i_2}^{p_2}\,.\,\sigma_{i_1}^{-p_1}\,.\, \sigma_{i_2}^{-p_2}\,.\,
\sigma_{i_1}^{p_1}\,. \sigma_{i_2}^{p_2}\,,$
which is impossible by lemma \ref{alpha-square-alpha-beta-alpha} (ii)
with $\alpha=\sigma_{i_1}^{p_1}\,.\,\sigma_{i_2}^{p_2}$
and $\beta = \sigma_{i_1}^{-p_1}\,.\, \sigma_{i_2}^{-p_2}$. This finishes 
the proof of proposition \ref{length-6}.
\end{proof}

Let us now prove the lemma \ref{alpha-square-alpha-beta-alpha}:
\begin{proof}
We first claim that if $L$, $\eta$ and $\delta$ are chosen as in the proposition
\ref{length-6} then we have 
\begin{equation}\label{eta-L-1}
\eta \leq {L\over 4000} 
\end{equation} 
and 
\begin{equation}\label{eta-L-2}
12 \eta +\argch \big(e^{12\eta}\big) \leq {1\over 4}\Big( 1- \frac{e^\delta}{\cosh(L/2)}\Big)\,.
\end{equation}
Proof of the claim.
By definition of $\eta$, (cf. (\ref{eta})), we have
$$
1000 \eta = \Big(1-\frac{\cosh(L/4)}{\cosh(L/2)}\Big)^{4}
$$
therefore
$$
1000 \eta < \frac{\cosh(L/2)-\cosh(L/4)}{\cosh(L/2)}
$$
and
$$
1000 \eta < \frac{\sinh(L/2) . L/4}{\cosh(L/2)} <\frac{L}{4},
$$
which proves the first inequality of the claim.
On the other hand, let $x\in \,]0,1[$, then 
$e^{x} \leq 1+2x \leq \cosh(2\sqrt{x})$. Choosing $x=12\eta$ we obtain, using 
the inequality $\eta < \frac{1}{1000}$, that 
$$
12\eta +\argch (e^{12\eta})\leq 12\eta + 2\sqrt{12\eta}< \frac{1}{4}\sqrt{1000\eta}
$$
therefore we get
$$
12\eta +\argch (e^{12\eta})\leq \frac{1}{4}\Big(1-\frac{\cosh(L/4)}{\cosh(L/2}\Big)^{2}
\leq \frac{1}{4}\Big(1-\frac{e^{\delta}}{\cosh(L/2)}\Big),
$$
which ends the proof of the claim.

{\it Proof of lemma \ref{alpha-square-alpha-beta-alpha} (i)}. 
Let us assume that $\gamma = a \alpha^2 b$ is $(L,\eta)$-straight, and $l_{\Sigma}(\gamma) = 6$.
Then by lemma \ref{technical-lemma} $\alpha$ and $\alpha^2 $ are $(L,3\eta)$-straight.
We then get with (\ref{eta-L-1})
$$
\rho(x_0, \alpha x_0) > (L -3 \eta)l_\Sigma (\alpha) > {L\over 2}\,. 
$$
On the other hand since $l(\alpha) \leq \delta$ and $\argch [e^\delta] = {L\over 4} <  {L\over 2}$,
we can apply lemma \ref{square} to $\alpha$ replacing $L$ by $L/2$ and get
$$
\rho(x_0, \alpha^2 x_0) < 2 \rho(x_0, \alpha x_0)- \Big( 1- \frac{e^\delta}{\cosh\big(L/2\big)}\Big)^2\,.   
$$
We then get by the choice of $\eta$, cf. (\ref{eta}), 
$$
\rho(x_0, \alpha^2 x_0) < (L- 3\eta)\,l_\Sigma (\alpha^2)\,
$$
which contradicts the fact that $\alpha^2$ is $(L,3\eta)$-straight
and concludes the proof of 
lemma \ref{alpha-square-alpha-beta-alpha} $(i)$.

{\it Proof of lemma \ref{alpha-square-alpha-beta-alpha} (ii)}.
Let us assume that $\gamma = a \alpha \beta \alpha b$ is $(L,\eta)$-straight, and $l_{\Sigma}(\gamma) = 6$.
The Lemma \ref{technical-lemma} says that $\alpha \beta \alpha$ is $(L,2\eta)$-straight
and that $\alpha \beta $ is  $(L,C'\,\eta)$-straight where 
$C'= 2 \frac{l_{\Sigma}(\alpha \beta \alpha)}{l_{\Sigma}(\alpha \beta)}$.
Since $\alpha \beta \alpha$ is $(L,2\eta)$-straight,
we have by triangle inequality
$$
(L-2\eta)l_{\Sigma}(\alpha \beta \alpha) \leq
2\rho(x_0, \alpha \, x_0) + L\,l_{\Sigma}(\beta) 
$$ 
and therefore 
$$
2\rho(x_0, \alpha \, x_0) \geq 
(L-2\eta)l_{\Sigma}(\alpha \beta \alpha) - L\,l_{\Sigma}(\beta)= L\,l_{\Sigma}(\alpha ^2)
-2\eta \, l_{\Sigma}(\alpha \beta \alpha)
$$
hence we obtain
$$
\rho(x_0, \alpha \, x_0) \geq L\,l_{\Sigma}(\alpha) -\eta \, l_{\Sigma}(\alpha \beta \alpha)\,,
$$
and since $l_{\Sigma}(\alpha)\leq 2$ and $l_{\Sigma}(\beta)\leq 2$,
we deduce that
$$
\rho(x_0, \alpha \, x_0) \geq (L-4 \eta) \,l_{\Sigma}(\alpha) \,,
$$
that is $\alpha$ is $(L,4\eta)$-straight.
We set $x_1=\alpha \beta \,x_0$,  $x_2=\alpha \beta \alpha \, x_0$ and 
$x_3= (\alpha \beta)^2 \, x_0 = \alpha \beta \alpha \beta \, x_0$.
We get, since $\alpha \beta \alpha$ is $(L,2\eta)$-straight,
\begin{eqnarray*}\label{I}
\rho(x_0, x_1) + \rho(x_1, x_2) -\rho(x_0, x_2)
&=& \rho(x_0, \alpha \beta \, x_0)+  \rho(x_0, \alpha \, x_0)- \rho(x_0, \alpha \beta \alpha \, x_0)\\
&\leq& L\big[l_\Sigma (\alpha \beta)+ l_\Sigma (\alpha)\big]
-(L-2 \eta)l_\Sigma (\alpha \beta \alpha) \\
&\leq& 12 \eta \,.  
\end{eqnarray*}
In the same way, since $\alpha \beta$ is $(L,C'\,\eta)$-straight with 
$C' =2 \frac{l_{\Sigma}(\alpha \beta \alpha)}{l_{\Sigma}(\alpha \beta)}$, we have 
\begin{eqnarray*}\label{II}
\rho(x_1, x_2) + \rho(x_2, x_3) -\rho(x_1, x_3)
&=& \rho(x_0, \alpha \, x_0)+  \rho(x_0, \beta \, x_0)- \rho(x_0, \alpha \beta  \, x_0)\\
&\leq& L\big[l_\Sigma (\alpha)+ l_\Sigma (\beta)\big]
-(L-C' \eta)l_\Sigma (\alpha \beta) \\
&\leq& 2 \eta l_\Sigma (\alpha \beta \alpha) \\ 
&\leq& 12 \eta \,. 
\end{eqnarray*}
We can therefore apply the lemma \ref{four-points} and get
\begin{eqnarray*}\label{III}
2 \rho(x_0, \alpha \beta  \, x_0) - \rho\big(x_0, (\alpha \beta)^2  \, x_0\big)
&\leq& \rho(x_0, \alpha \beta  \, x_0) + \rho(x_0, \alpha \, x_0)+  \rho(x_0, \beta \, x_0)
-\rho\big(x_0, (\alpha \beta)^2  \, x_0\big)\\
&=&
\rho(x_0, x_1) + \rho(x_1, x_2) +\rho(x_2, x_3) - \rho(x_0, x_3)\\
&\leq& 
\Big(1+ \frac{\rho(x_0, \beta \, x_0)}{\rho(x_0, \alpha \, x_0)}\Big)\big( 12 \eta + Argcosh [e^{12 \eta}]\big)\\
&\leq&
\Big(1+ \frac{L\,l_{\Sigma}(\beta)}{(L-4\eta)l_{\Sigma}(\alpha)}\Big)\,. 
{1\over 4}\Big(1-\frac{e^\delta}{\cosh(L/2)}\Big)^2 \,, 
\end{eqnarray*}
the last inequality coming from (\ref{eta-L-2}) and the fact that
$\alpha$ is $(L,4\eta)$-straight.
From (\ref{eta-L-1}) we therefore get 
\begin{equation}\label{hyp1}
\rho\big(x_0, (\alpha \beta)^2  \, x_0\big)
\geq 2 \rho(x_0, \alpha \beta  \, x_0) - \Big(1-\frac{e^\delta}{\cosh(L/2)}\Big)^2 \,.
\end{equation}
On the other hand we have seen that  $\alpha \beta$ is
$(L,C'\,\eta)$-straight with $C'= 2 \frac{l_{\Sigma}(\alpha \beta \alpha)}{l_{\Sigma}(\alpha \beta)}$,
so that
$$
\rho(x_0, \alpha \beta  \, x_0) \geq (L - C'\eta)\,l_{\Sigma}(\alpha \beta)
\geq 2L - 2\eta \,l_{\Sigma}(\alpha \beta \alpha)\,,
$$
and since $l_{\Sigma}(\alpha \beta \alpha)\leq 6$ the above inequality gives with (\ref{eta-L-1})
\begin{equation}\label{hyp2}
\rho(x_0, \alpha \beta  \, x_0) \geq L\,.
\end{equation}
By assumption, since $l_{\Sigma}(\alpha \beta) \leq 4$, the displacement of $\alpha \beta$ satisfies
$l(\alpha \beta) \leq \delta$, and with (\ref{hyp2}) we can apply the lemma \ref{square} to get 
$$
\rho\big(x_0, (\alpha \beta)^2 x_0\big) 
\leq  2 \rho(x_0, \alpha \beta  \, x_0) -\Big(1-\frac{e^\delta}{\cosh(L/2)}\Big)^2
$$
which contradicts (\ref{hyp1}). This concludes the proof of the lemma 
\ref{alpha-square-alpha-beta-alpha} and the proposition \ref{length-6}.
\end{proof}

\section{Mapping the Cayley graph of $G$ into $X$.}
Let $G$ be a finitely generated discrete group of isometries of $(X,g)$ a 
Cartan Hadamard manifold of sectional curvature $-a^2 \leq K \leq -1$. 
We consider $S$ a finite generating set of $G$ and 
the Cayley graph $\mathcal{G}_{S}$ of $G$ associated to $S$. We define a distance 
$d_S$ on $\mathcal{G}_{S}$ in the following way: each edge is isometric to the segment $[0,1] \subset \mathbb{R}$
and the distance $d_S(\gamma, \gamma')$ between two vertices $\gamma$, $\gamma'$ of $\mathcal{G}_{S}$
is the word distance $d_S(\gamma, \gamma')= l_S(\gamma^{-1} \gamma')$.
The group $G$ acts by isometries on $(\mathcal{G}_{S},d_S)$ and on $(X,g)$.
The goal of this section is to construct for each number $c$ large enough an equivariant 
map $f_c : \mathcal{G}_{S} \to X$ such that $f_c$ is Lipschitzian of Lipschitz constant $c$.

\subsection{Poincar\'e series, measures and convexity.}
We first consider the Poincar\'e series,
\begin{equation}\label{poincare}
P_c(s,x,y)= \Sigma_{\gamma \in G} e^{-cd_S (s,\gamma)}\cosh\big[\rho(x,\gamma y)\big]
\end{equation}
where $c\in \mathbb{R}_{+}$, $s\in \mathcal{G}_S$ and $x, y \in X$.
\begin{lemma}\label{equiv-poincare}
For all $s\in \mathcal{G}_{S}$, $x, y, x_0, y_0 \in X$, $c\in \mathbb R$ and $\gamma_0 \in G$ we have

$(i) $ $P_c(\gamma _0 s,\gamma _0 x,y)=P_c(s,x,y)$ 

$(ii)$ $P_c(s,x,y) \leq P_c(s,x_0,y_0)\,e^{\rho(x_0,x) + \rho(y_0,y)}$. 

\noindent 
In particular the convergence of the series is independant of the choice of the points $x, y \in X$.
\end{lemma}
\begin{proof}
The equivariance property of the Poincar\'e series is straightforward.
On the other hand by triangle inequality we have
\begin{eqnarray*}
P_c(s,x,y)
&=& \Sigma_{\gamma \in G} e^{-cd_S (s,\gamma)}\cosh\big[\rho(x,\gamma y)\big] \\
&\leq& \Sigma_{\gamma \in G} e^{-cd_S (s,\gamma)}\cosh\big[\rho(x_0,\gamma y_0)+\rho(x_0,x) + \rho(y_0,y)\big]
\end{eqnarray*}
hence we get
$$
P_c(s,x,y) \leq P_c(s,x_0,y_0)\,.e^{\rho(x_0,x) + \rho(y_0,y)}\,.
$$
\end{proof}
The critical exponent of this series is defined as 
$$
c_0 =: \inf \{ c>0 \quad / \quad P_c(s,x,y)< \infty\}\,.
$$
Let $x_0$ be the point of $X$ such that $L(S)= \max _{i} \{\rho(x_0,\sigma_i x_0) \}$.
By the triangle inequality we have for all $\gamma \in \Gamma$,
$\rho(x_0, \gamma x_0) \leq L(S) \,l_S(\gamma)$, therefore 
$$
P_c(e,x_0,x_0)\leq \Sigma _{\gamma \in \Gamma}e^{\big(c-L(S)\big) \,l_S(\gamma)}\,.
$$
On the other hand, by definition of $\Ent_S \Gamma$, we have
$ \Sigma _{\gamma \in \Gamma}e^{-t\,l_S(\gamma)} < \infty $ for all $t > \Ent_S \Gamma$,
hence we have proved that
\begin{equation}
c_0 \leq \Ent_S \Gamma + L(S)\,.
\end{equation}

{\bf We now consider untill the end of this section a $c\in \mathbb{R}_{+}$ such that $P_c(s,x,y)<\infty$.}

Let us choose a probability measure $\mu$ with smooth density and compact support on $X$.
For each $s\in \mathcal{G}_S$ let us define the measure on $X$
\begin{equation}\label{patterson}
\mu_s^c = \Sigma _{\gamma \in G} e^{-cd_S(s,\gamma)}\gamma_{\ast}\mu
\end{equation}
and the function $\mathcal{B}^{c} :G \times X \to \mathbb{R}$,
\begin{equation}\label{B}
\mathcal{B}^{c}(s,x) = \int_X \cosh\big[\rho(x,z)\big]d\mu_s^c(z).
\end{equation}
In the following lemmas \ref{convexite}, \ref{formules} and corollary \ref{B-convex} we
show that $ x\to \mathcal{B}^{c}(s,x)$ is a strictly convex $C^2$ function
such that 
$$
\lim_{x\to \infty}\mathcal{B}^{c}(s,x)= +\infty\,.
$$

\begin{lemma}\label{convexite}
Let $c$ be such that $P_c(s,x,y) <\infty$. For all
$s\in\mathcal{G}_S$ and $x\in X$, we have $\mathcal{B}^{c}(s,x)<\infty$.
Moreover,
the function $x\to \mathcal{B}^{c}(s,x)$ is stricly convex and  
$\lim_{x\to \infty} \mathcal{B}^{c}(s,x)= +\infty$.
\end{lemma}
\begin{proof}
By definition of $\mu_s^c$,
$$
\mathcal{B}^{c}(s,x) = \int_X   \Sigma _{\gamma \in G} e^{-cd_S(s,\gamma)}\cosh\big[\rho(x,\gamma z)\big]d\mu(z)
= \int_X P_c(s,x,z)d\mu (z)\,,
$$
so we get $\mathcal{B}^{c}(s,x)<\infty$
by lemma \ref{equiv-poincare}  $(ii)$ since the support of $\mu$ is compact.
For any geodesic $c(t)$ and $z$ in $X$, $t\to d(c(t),z)$ is a convex function since 
$(X,g)$ has negative sectional curvature, therefore $t\to \cosh\big[\rho(c(t),z)\big]$
is stricly convex and so is $x\to \mathcal{B}^{c}(s,x)=\int_X \cosh\big[\rho(x,z)\big]d\mu_s^c(z)$.
On the other hand we have  
$$
\mathcal{B}^{c}(s,x)=\int_X \cosh\big[\rho(x,z)\big]d\mu_s^c(z)\geq {1\over 2}e^{\rho(x,x_0)}
\int_X e^{-\rho(x_0,z)} d\mu^c_s(z)\,,
$$
so $\mathcal{B}^{c}(s,x)\to +\infty$  whenever $x$ tends to infinity in $X$.
\end{proof}

In the above lemma we proved that  $x \to \mathcal{B}^c(s,x)$ is a convex function which
tends to $+\infty$ when $x$ tend to infinity. We shall now prove 
that $x \to \mathcal{B}^c(s,x)$ is a stricly convex $C^2$ function. We will also give estimates 
of the second derivative of $x\to \mathcal{B}^c(s,x)$.
\begin{lemma}\label{formules}
Let $c$ be such that $P_c(s,x,y) < \infty$. 
The function 
$x\to\mathcal{B}^c(s,x)$ is $C^2$ and for any $s\in \mathcal{G}_S$, $x\in X$ and
any tangent vectors $v, w\in T_x X$ we have
$$
d\mathcal{B}^c(s,x)(v) = \int_X d\rho(x,z)(v) \sinh\big[\rho(x,z)\big]d\mu_s^c(z)
$$ 
and
$$
Dd\mathcal{B}^c(s,x)(v,w)=
$$
$$
\int_X \Big(\sinh\big[\rho(x,z)\big] Dd\rho(x,z)(v,w) +
\cosh\big[\rho(x,z)\big] d\rho(x,z) \otimes d\rho(x,z)(v,w)\Big)d\mu_s^c(z)
$$
\end{lemma}
\begin{proof}
Let $v\in T_x X$ be a unit tangent vector at a point $x\in X$. 
For each point $z \neq x$ in $X$, we have
$$
d\,\big(\cosh\big[\rho(x,z)\big]\big)(v)
=d\rho(x,z)(v) \,\sinh\big[\rho(x,z)\big]\,,
$$
hence we get
\begin{equation}\label{derive}
|d\,\big(\cosh\big[\rho(x,z)\big]\big)(v)| =
|d\rho(x,z)(v) \,\sinh\big[\rho(x,z)\big]| 
\leq  \cosh\big[\rho(x,z)\big]\,,
\end{equation}
therefore, $\cosh\big[\rho(x,z)\big] \leq 2\cosh\big[\rho(x_1,z)\big]$ for $x$
in a sufficiently small neighbourhood of an arbitrary point $x_1$.
Since $ z\to 2\cosh\big[\rho(x_1,z)\big]$ is $\mu^c_s$-integrable,
we can differentiate $x \to \mathcal{B}^c(s,x)$ applying Lebesgue derivation theorem.
Let us now compute the second derivative.
Let $v, w \in T_x X$ be unit tangent vectors at $x\in X$. Let $\alpha(t)$ the geodesic
such that $\alpha(0)=x$ and $\alpha'(0) = v$. We denote 
$W(t)$ the parallel vector field along $\alpha$ such that $W(0) = w$ and write
$\rho_{(z,\alpha(t))}$ instead of $\rho(z,\alpha(t))$. 
Let us denote
$$
h(t,z)={1\over t}\Big(d\rho_{(z,\alpha(t))}(W(t))\,\sinh\big[\rho_{(z,\alpha(t))}\big]
-d\rho_{(z,\alpha(0))}(W(0))\,\sinh\big[\rho_{(z,\alpha(0))}\big]\Big).
$$
When $z\neq x$ we have
\begin{equation}\label{limh(t,z)}
\lim_{t\to 0} h(t,z)=\sinh\big[\rho_{(x,z)}\big] Dd\rho_{(x,z)}(v,w) +
\cosh\big[\rho_{(x,z)}\big] d\rho_{(x,z)} \otimes d\rho_{(x,z)}(v,w)\,.
\end{equation}
We will write 
\begin{equation}\label{h_0}
h_0(z)=:\lim_{t\to 0} h(t,z).
\end{equation}
The formula which gives $Dd\mathcal{B}^c(s,x)(v,w)$ in lemma \ref{formules} is equivalent
to 
\begin{equation}\label{DdB}
Dd\mathcal{B}^c(s,x)(v,w)=\int_X h_0(z)d\mu_s^c(z)
\end{equation}
and will be a consequence
of the dominated convergence theorem of Lebesgue with the existence of a 
$\mu_s^c(z)$- integrable function $H(z)$ such that for 
any $z\notin \alpha([0,t])$ then $h(t,z) \leq H(z)$.
Let us now prove the existence of such a function $H$.
For each $z \notin \alpha ([0,t])$ we have
\begin{eqnarray*}
|h(t,z)| 
& \leq & \sup_{t'\in [0,t]}
\Big|\sinh\big[\rho_{(z,\alpha(t'))}\big] Dd\rho_{(z,\alpha(t'))}(\dot{\alpha}(t'),W(t')) +\dots \\
& & \dots 
+ \cosh\big[\rho_{(z,\alpha(t'))}\big] d\rho_{(z,\alpha(t'))} 
\otimes d\rho_{(z,\alpha(t'))}(\dot{\alpha}(t'),W(t'))\Big|\,.
\end{eqnarray*}
Since the curvature of $(X,g)$ satisfies  $-a^2 \leq K \leq -1$, the Rauch comparison theorem shows that 
for each $x,y \in X$
$$
Dd\rho_{(x,y)} \leq a \frac{\cosh[a\rho_{(x,y)}]}{\sinh[a\rho_{(x,y)}]} 
\Big(g-d\rho_{(x,y)} \otimes d\rho_{(x,y)}\Big)\,,
$$
hence we get from the previous inequality
\begin{eqnarray*}
|h(t,z)| 
&\leq &
\Big[a \sinh\big[\rho_{(z,\alpha(t'))}\big]
\frac{\cosh[a\rho_{(z,\alpha(t'))}]}{\sinh[a\rho_{(z,\alpha(t'))}]} 
\Big(g-d\rho_{(z,\alpha(t'))} \otimes d\rho_{(z,\alpha(t'))}\Big)+\dots\\
& & \dots 
+\cosh\big[\rho_{(z,\alpha(t'))}\big]
 d\rho_{(z,\alpha(t'))} \otimes d\rho_{(z,\alpha(t'))}\Big](\dot{\alpha}(t'),W(t')) \,.
\end{eqnarray*}
But since $a \geq 1$ the concavity of the function $\tanh$ on $\mathbb R _+$ gives 
$$
\frac{a}{\tanh a\rho} \geq {1\over \tanh \rho}
$$
therefore we get
\begin{eqnarray}\label{estimation}
|h(t,z)| 
&\leq &
a \sinh\big[\rho_{(z,\alpha(t'))}\big]
\frac{\cosh[a\rho_{(z,\alpha(t'))}]}{\sinh[a\rho_{(z,\alpha(t'))}]}\,.
\end{eqnarray}
Since ${\sinh \rho \leq {1\over a} \sinh a\rho}$ by convexity of $\sinh$, we then get that 
$|h(t,z)| \leq H(z)$ from \ref{estimation}
for all $|t| \leq {1\over a}$ and all $z\notin \alpha([0,t])$  
where
$$
H(z) = \left\{
\begin{array}{ll}
a{\cosh1 \over \sinh1} \sinh [\rho_{(z,\alpha(0))}+1], \quad \rho_{(z,\alpha(0))}\geq {2\over a} \\
\cosh [a \rho_{(z,\alpha(0))}+1],\quad \rho_{(z,\alpha(0))} < {2\over a}
\end{array}
\right.
$$
This concludes the proof of lemma \ref{formules}.
\end{proof}
The above lemma \ref{formules} has the following 
\begin{corollary}\label{B-convex}
Under the assumptions of lemma \ref{formules} we have 
$$
Dd\mathcal{B}^c \geq \mathcal{B}^c\,.g\,,
$$
in particular, $\mathcal{B}^c$ is strictly convex.
\end{corollary}
\begin{proof}
Since the sectional curvature of $(X,g)$ satisfies $K \leq -1$ Rauch's theorem shows
that 
$$
Dd\rho \geq {1\over \tanh \rho}\Big(g- d\rho \otimes d\rho\Big)\,.
$$
From this inequality and lemma \ref{formules} we therefore get, for all $x\in X$
and any unit tangent vector $v\in T_x X$,
$$
Dd\mathcal{B}^c (v,v) \geq \Big(\int_X \cosh [\rho_{(z,x)}] d \mu_s^c(z) \Big)g(v,v) 
= \mathcal{B}^c(x)\,. g(v,v)\,.
$$
\end{proof}
\subsection{Construction of Lipschitzian maps $f_c : \mathcal{G}_S \to X$.}
So far we have shown that for any $s\in \mathcal{G}_S$ the function
$x \to \mathcal{B}^c(s,x)$ is strictly convex and tends to $+\infty$ when
$x$ tend to infinity. 
We then can define the map $f_c :\mathcal{G}_S \to X$ as follows.
For $s\in \mathcal{G}_S$ we define $f_c(s)$ as the unique point $x\in X$ 
which achieves the strict minimum of the function
$x \to \mathcal{B}^c(s,x)$. 
The end of this section is devoted to proving the following
\begin{proposition}\label{lipschitz}
Let $c$ be such that $P_c(s,x,y) < \infty$.
Let $f_c :(\mathcal{G}_S,d_S) \to (X,g)$ which associates to $s\in \mathcal{G}_S$
the unique point $x\in X$ 
which achieves the unique minimum of the function
$x \to \mathcal{B}^c(s,x)$. 
Then, $f_c$ is Lipschitzian of Lipschitz constant equal to $c$.
\end{proposition}
The proof of the proposition \ref{lipschitz} relies on 
the following technical lemmas.
\begin{lemma}\label{s-derivabilite}
Let $c$ be such that $P_c(s,x,y) < \infty$.
For all $x \in X$ and all tangent vector 
$v\in T_x X$ the function $\alpha :s \to d\mathcal{B}^c(s,x)(v)$ is differentiable at each point
$s\in \mathcal{G}_S$ distinct from a vertex or a middle point of an edge.
Moreover, for such an $s$ we have 
\begin{eqnarray*}
\alpha'(s)&=&
-c \int_X d\rho_{(x,z)}(v) \sinh \big[\rho(x,z)\big]
\Sigma _{\gamma\in G}
{d\over ds} \Big(d_S(s,\gamma)\Big)
e^{-c d_S(s,\gamma)} d(\gamma _{\ast} \mu)(z)\\
\end{eqnarray*}
\end{lemma} 
\begin{proof}
Let us denote by $[g,g']$ the edge containing $s$ and parametrize it by $t\in [0,1]$.
We first observe that for all $\gamma \in G$ then 
$$
d_S(s,\gamma)= \min \big[ d_S(g,\gamma)+t, d_S(g',\gamma) +1-t\big]\,,
$$
therefore $s\to d_S(s,\gamma)$ is differentiable at each $s\in ]g,g'[$ distinct of the middle point of $]g,g'[$.
On the other hand we have by lemma \ref{formules}
$$
d\mathcal{B}^c(s,x)(v)= \int_X d\rho(x,z)(v) \sinh\big[\rho(x,z)\big]d\mu_s^c(z)\,,
$$
so that we can write
\begin{eqnarray*}
&{1\over t} \Big(\alpha(s+t)-\alpha(s)\Big)=&\\
& \Sigma _{\gamma\in G} \int_X d\rho_{(x,\gamma z)}(v) \sinh \big[\rho(x,\gamma z)\big]
\,. {1\over t} \Big[
e^{-c d_S(s+t,\gamma)} - e^{-c d_S(s,\gamma)}\Big] d\mu(z)&\,,
\end{eqnarray*}
where we have identified the point $s$ in the edge $[g,g']$ with its parameter.
Let us observe that for $|t|$ small enough,
$$
\Big|{1\over t}\Big[
e^{-c d_S(s+t,\gamma)} - e^{-c d_S(s,\gamma)}\Big] \Big|\leq 2c \,e^{-c d_S(s,\gamma)}\,,
$$
and that
$$
2c \,\Sigma _{\gamma\in G} \int_X \big|d\rho_{(x,\gamma z)}(v)\big| \sinh \big[\rho(x,\gamma z)\big]
\,.
e^{-c d_S(s,\gamma)} d\mu(z) < \infty \,,
$$
hence if $s\in \mathcal{G}_S$ is distinct from a vertex or a middle point of an edge we get
\begin{eqnarray*}
&\lim _{t\to 0}{1\over t} \Big(\alpha(s+t)-\alpha(s)\Big)=&\\
&-c \int_X d\rho_{(x,z)}(v) \sinh \big[\rho(x,z)\big]
\Sigma _{\gamma\in G}
{d\over ds} \Big(d_S(s,\gamma)\Big)
e^{-c d_S(s,\gamma)} d(\gamma _{\ast} \mu)(z)&
\end{eqnarray*}
by Lebesgue's theorem.
\end{proof}
\begin{lemma}\label{derivative-estimate}
Let $c$ be such that $P_c(s,x,y) < \infty$.
Let $s_0\in \mathcal{G}_S$ be a point distinct from a vertex or a middle point of an edge,
and $u$ a unit vector tangent at $s_0$ to the edge containing $s_0$. 
Then, we have $||df_c(u)|| \leq c$.
\end{lemma}
\begin{proof}
Let us fix a smooth moving frame $\{E_1,\dots, E_n\}$ of $TX$ and define the 
function
$\Phi  : X \times \mathcal{G}_S \to \mathbb{R}^n$
by 
$$
\Phi(x,s) = \big( d\mathcal{B}^c(s,x)(E_1), \dots ,  d\mathcal{B}^c(s,x)(E_n)\big)\,.
$$
By definition, the point $f_c(s)$ is characterized by the implicit equation
$$
\Phi(f_c(s),s)=0\,,
$$
or equivalently,
$$
d\mathcal{B}^c(s,f_c(s))=0\,.
$$
For all $x\in X$ and $s\in \mathcal{G}_S$ in a neighbourhood of $s_0$ the function $\Phi$
is differentiable by lemma \ref{formules} and \ref{s-derivabilite}. Moreover
since $x=f_c(s)$ is a critical point of the function $x\to \mathcal{B}^c(s,x)$, we have, for $j=1,\dots,n$,
$$
{\partial \Phi \over \partial x}(f_c(s),s) (E_j)
= \big( Dd\mathcal{B}^c(s,f_c(s))(E_j,E_1), \dots ,  Dd\mathcal{B}^c(s,f_c(s))(E_j,E_n)\big)\,,
$$
thus $ {\partial \Phi \over \partial x}(f_c(s),s)$ is invertible by corollary \ref{B-convex}. 
By the implicit function theorem, the function $f_c$ is then differentiable at $s$ in
a neighbourhood of $s_0$ and we have, if $u$ is a unit vector tangent at $s_0$ to the edge containing $s_0$
and $v$ a tangent vector in $T_{f_c(s)} X$,
\begin{equation}\label{derivee}
Dd\mathcal{B}^c({s_0},f_c(s_{0}))(df_c(u),v))=
-{d\over ds}_{/s=s_0}d\mathcal{B}^c(s,f_c(s_0))(v)\,.
\end{equation}
From corollary \ref{B-convex} and lemma \ref{s-derivabilite} we obtain, setting $v=\frac{df_c(u)}{||df_c(u)||}$
\begin{eqnarray*}
&g(df_c(u),v) \,\mathcal{B}^c({s_0},f_c(s_{0})) \leq &\\
&c \int_X \big|d\rho_{(f_c(s_0),z)}(v)\big| \sinh \big[\rho(f_c(s_0),z)\big]
\Sigma _{\gamma\in G}
\big|{d\over ds}_{/s=s_0} \Big(d_S(s,\gamma)\Big)\big|
e^{-c d_S(s_0,\gamma)} d(\gamma _{\ast} \mu)(z)&
\end{eqnarray*}
therefore
\begin{equation}\label{major}
\big|g(df_c(u),v) \,\mathcal{B}^c\big({s_0},f_c(s_0)\big)\big| \leq 
c \int_X  \sinh \big[\rho(f_c(s_0),z)\big]d \mu^c_{s_0}(z)\,.
\end{equation}
hence
$$
||df_c(u)|| \leq c\, \frac{\int_X  \sinh \big[\rho(f_c(s_0),z)\big]d \mu^c_{s_0}(z) }
{\int_X  \cosh \big[\rho(f_c(s_0),z)\big]d \mu^c_{s_0}(z) }
\leq c \,,
$$
which completes the proof of lemma \ref{derivative-estimate}.
\end{proof} 
The proposition \ref{lipschitz} follows then from the 
\begin{corollary}\label{the-end}
Let $c$ be such that $P_c(s,x,y) < \infty$.
The map $f_c$ is Lipschitzian of Lipschitz constant equal to $c$.
\end{corollary}
\begin{proof}
Let us consider a segment $[s_1, s_2 ] \subset \mathcal{G}_S$ which 
contains no edges nor middle points. It directly follows from
lemma \ref{derivative-estimate} that 
\begin{equation}\label{c-lipschitz}
\rho\big(f_c(s_1),f_c(s_2)\big)\leq c\, d_S(s_1, s_2)\,.
\end{equation}
We now want to extend the inequality (\ref{c-lipschitz}) for all points
$s_1, s_2 \in \mathcal{G}_S$. 
For that purpose we first consider
a segment $[s_1, s_2 ] \subset \mathcal{G}_S$ where $s_1$ is a midpoint of an edge $e$
and $s_2$ a vertex of the same edge $e$ and the inequality (\ref{c-lipschitz}) for these
points $s_1,s_2$ derives from the continuity of $f_c$ at $x_1$ and $x_2$.
The corollary \ref{the-end} will then follow from the fact that any segment $[s_1, s_2 ] \subset \mathcal{G}_S$
can be decomposed in a finite sequence of adjacent intervals $[y_1^k, y_2^k]$ where
$y_1^k$ is a midpoint and $y_2^k$ a vertex of one same edge or the other way around.
Let us prove the continuity of $f_c$ at a vertex or a midpoint $s$ of an edge.
Given such a point $s$, let $\{s_k\}_{k\in\mathbb N}$ be a sequence converging 
to $s$ and staying in a single mid-edge containing $s$. The sequence $x_k=:f_c(s_k)$ is a Cauchy sequence in $X$
by (\ref{c-lipschitz}) whose limit is a point $x=\lim _k x_k$. 
We want to prove that $f_c(s)=x$.
For all $z\in X$ and $k\in \mathbb N$ we have 
\begin{equation}\label{inegalite}
\mathcal{B}^c(s_k,z) \geq \mathcal{B}^c(s_k,x_k)
\end{equation}
by  definition of $x_k=f_c(s_k)$.
We claim that $\lim _k \mathcal{B}^c(s_k,x_k) =\mathcal{B}^c(s,x)$
and that $\lim _k \mathcal{B}^c(s_k,z) =\mathcal{B}^c(s,z)$ . Assuming the claim and passing 
to the limit in (\ref{inegalite}) when $k$ tends to infinity gives for all $z\in X$,
\begin{equation}\label{inegalite*}
\mathcal{B}^c(s,z) \geq \mathcal{B}^c(s,x)\,,
\end{equation}
therefore $x=f_c(s)$.
Let us prove the claim. By definition (\ref{B}) and (\ref{patterson}), we have 
\begin{eqnarray*}
\mathcal{B}^{c}(s_k,x_k) &=& \int_X \cosh\big[\rho(x_k,z)\big]d\mu_{s_k}^c(z)\\
&=& \int_X   \Sigma _{\gamma \in G} e^{-cd_S(s_k,\gamma)}\cosh\big[\rho(x_k,\gamma z)\big]d\mu(z)\,,
\end{eqnarray*}
Since 
$e^{-cd_S(s_k,\gamma)}\cosh\big[\rho(x_k,\gamma z)\big]
\leq e^{c}\,e^{-cd_S(s,\gamma)}\cosh\big[\rho(x,\gamma z)+1\big]$
for $k$ large enough, we get $\lim _k \mathcal{B}^c(s_k,x_k) =\mathcal{B}^c(s,x)$ by the 
Lebesgue's theorem. By the way we get $\lim _k \mathcal{B}^c(s_k,z) =\mathcal{B}^c(s,z)$ 
which concludes the proof of the claim, the corollary \ref{the-end} and the proposition \ref{lipschitz}.
\end{proof}

\section{Algebraic Entropy and $\eta$-straight isometries.}

Let $G$ be a finitely generated discrete group of isometries of $(X,g)$ whose sectional curvature satisfies 
$-a^2 \leq K_g\leq -1$, and $S=\{\sigma_1,\dots, \sigma_p\}$ be a finite generating set.

We assume that the minimal displacement 
$L(S)= \inf_{x\in X}\max_{i=1,\dots p} \rho(x,\sigma_i x)$ of $S$ 
(cf. definition \ref{0})
satisfies $L(S)>0$. 
By lemma \ref{x0} there exist a point $x_0 \in X$ such that 
$$L(S)=\inf_{x\in X}\max_{i\in\{ 1,\dots,p\}}\rho(x,\sigma_ix)=\max_{i\in\{ 1,\dots,p\}}\rho(x_0,\sigma_ix_0)\,.$$ 
The goal of this section is to prove that if all elements 
of $G$ are ``almost non $\eta$-straight'' for some $\eta$ 
such that $L(S)>\eta>0$, then the entropy
of $G$ with respect to $S$ is bounded below by $\eta$. 
By ``almost non $\eta$-straight'' elements we mean isometries $\gamma$ such that
$\rho (x_0,\gamma x_0)\leq (L(S)-\eta )l_S(\gamma )+D$,
for some positive number $D$.
\begin{theo}\label{theo}
Let $G$ be a finitely generated discrete group of isometries of $(X,g)$ whose sectional curvature satisfies 
$-a^2 \leq K_g\leq -1$, and $S=\{
\sigma_1,\dots, \sigma_p\}$ be a finite generating set of $G$
with 
$L(S)=\inf_{x\in X}\max_{i\in\{ 1,\dots,p\}}\rho(x,\sigma_ix)=\max_{i\in\{ 1,\dots,p\}}\rho(x_0,\sigma_ix_0) >0\,.$
Let us assume
that there exist $D\geq0$ and $\eta$, $0 < \eta < L(S)$, such that
for all $\gamma\in G$, 
\begin{equation}\label{straight}
\rho (x_0,\gamma x_0)\leq (L(S)-\eta )l_S(\gamma )+D\,,
\end{equation}
then $\Ent_S (G)\geq \eta$.
\end{theo}

\begin{proof}
The proof relies on the construction made in section 3
of an equivariant Lipschitzian map of Lipschitz constant $c > \Ent_S (G) + L(S) - \eta$. 

Let us prove that under the assumption (\ref{straight}) then for any $c > \Ent_S (G) + L(S) - \eta$ we have $P_c(s,x,y) <\infty$.
By triangle inequality we have 
$$
e^{-cd_S(s,\gamma)} \leq e^{cd_S(s,e)}e^{-cd_S(\gamma, e)}\,,
$$
and for any $x_0 \in X$
$$
\cosh\big[\rho(x,\gamma y)\big] \leq e^{\rho(x,\gamma y)}\leq
e^{\rho(x,x_0)+\rho(x_0,y)+\rho(x_0,\gamma x_0)}\,.
$$
Therefore for $x_0$, $D$ and $\eta$ chosen such that (\ref{straight}) holds, we get
$$
P_c(s,x,y) \leq 
e^{D+c\,d_S(e,s) +\rho(x,x_0)+\rho(x_0,y)} \Sigma _{\gamma \in G}e^{[L(S)-\eta -c]d_S(e,\gamma)}\,,
$$
and so
$P_c(s,x,y) <\infty$ for each $c > \Ent_S (G) + L(S) - \eta$.

Hence by section 3, proposition \ref{lipschitz} 
there exists an equivariant Lipschitzian map 
$f_c : (\mathcal{G}_{S},d_S) \to (X,g)$ 
of Lipschitz constant $c$ 
for any $c > \Ent_S (G) + L(S) - \eta$. 
We consider the point $x= f_c(e)$, where $e$ is the neutral element of $G$. By definition
of $L(S)$, there is a $\sigma_i \in S$ such that $\rho(x, \sigma_i x) \geq L(S)$.
Therefore, by equivariance, 
$$
\rho\big(f_c(e), \sigma_i (f_c (e))\big)=
\rho\big(f_c(e), f_c(\sigma_i (e))\big) \geq L(S)\,.
$$  
On the other hand, since $f_c$ is $c$-Lipschitzian we have
$$
\rho\big(f_c(e), f_c(\sigma_i (e))\big) \leq c\,d_S(e,\sigma_i (e)) = c\,.
$$
The two above inequalities give 
$$
c\geq L(S)
$$
and since $c$ is any number such that $c > \Ent_S (G) + L(S) - \eta$,
we get $\Ent_S (G)\geq \eta$.
\end{proof}
\section{Proof of the main theorem.}
In this section we shall first prove that the entropy of a group with respect to
a set of two generators with displacement $L >0$ is bounded below. Then we shall prove the 
main theorem.  
\begin{proposition}\label{entropy-bounded}
Let $(X,g)$ be a Cartan-Hadamard manifold whose sectional curvature satisfies $\-a^2 \leq K_g\leq -1$
and $G$ a discrete group of isometries of $(X,g)$
generated by two isometries
$\{\sigma_1, \sigma _2\}$. Let us assume 
$$
L =\inf_{x\in X}\max\{\rho(x,\sigma_1 x),\rho(x,\sigma_2 x)\} >0\,.
$$
Then the entropy of $G$ relatively to the set of generators
$\Sigma= \{\sigma_1, \sigma _2 \}$ satisfies
$$ ent_{\Sigma} G \geq \min\Bigg [ {\log (\ch ( {L\over 4}))\over 5+\log (\ch ({L\over 4}) )}
{\log 2\over 6}, {1\over 1000}\Big (1-{\ch ({L\over 4})\over \ch ({L\over 2})}\Big )^4  \Bigg ]\,.$$
\end{proposition}
\begin{proof}
Let $\delta=\log \cosh ({L\over 4})$. The proof divides into two cases.
In the first case we can find two elements in $G$ of bounded length $l_{\Sigma}$
which are hyperbolic with distinct axes and displacement larger than $\delta$.
In that case, a ping-pong argument shows that 
the semigroup generated by these two elements (or their inverse) is free
with corresponding entropy bounded below by a constant depending on $\delta$. 
In the second case, when we cannot find such a free semigroup, then we can show that all elements of $G$ are 
almost non
$\eta$-straight for some $\eta=\eta (\delta, L)$ and we will conclude by theorem \ref{theo}.  

\noindent
{\bf Case 1. }  There exists an element $\gamma\in G$ of algebraic length 
$l_{\Sigma}(\gamma) \leq 4$ whose displacement $l(\gamma)$ in $X$ is bounded below $\l(\gamma) > \delta$.

\noindent
{\bf Case 2. }  The displacement of all elements $\gamma\in G$ of algebraic length 
$l_{\Sigma}(\gamma) \leq 4$ satisfies $\l(\gamma) \leq \delta$.

\noindent
In the case 1, let us consider an element $\gamma \in G$ of algebraic length 
$l_{\Sigma}(\gamma) \leq 4$ and whose displacement $l(\gamma)$ in $X$ satisfies $\l(\gamma) > \delta$. 
We note that $\gamma$ is then an hyperbolic isometry of $X$.
Since $G$ is not virtually nilpotent 
one of the generators $\sigma_1$ or $\sigma_2$, say $\sigma_1$  does not preserve the axis of $\gamma$. 
Indeed if both $\sigma_1$ and  $\sigma_2$ were preserving the axis of $\gamma$, then $G$
would preserve the axis of $\gamma$ and hence would be virtually abelian by lemma \ref{fixed-point} (ii), contradiction.
Then, if $(\theta, \eta)$ are the endpoints of the axis of $\gamma$, 
$\sigma _1 (\left\{\theta, \eta \right\}) \cap  \left\{\theta, \eta \right\} = \emptyset$
by the proof of lemma \ref{virtually-nilpotent}, a).
We can then apply the effective ping-pong lemma proved in the appendix to 
the two hyperbolic elements $\gamma$ and 
$\sigma_1\gamma\sigma_1^{-1}$ which have disjoint fixed-point-sets. 
This shows that the algebraic entropy of the subgroup generated by 
$\gamma$ and $\sigma_1\gamma\sigma_1^{-1}$ is bounded below by 
${\delta\over 5+\delta}\log 2$. We then deduce that,
$$\Ent_\Sigma (\Gamma )\geq  {\delta\over 5+\delta}{\log 2\over 6}\,.$$

In the case 2, proposition \ref{length-6} tells that all elements $\gamma \in G$ of length 
$l_\Sigma (\gamma) =6$ are not $(L,\eta)$-straight where $\eta$ is given by (\ref{eta}),
$\eta = 10^{-3}\Big(1- \frac{\cosh({L\over4})}{\cosh({L\over2})}\Big)^4\,.$
Then every element $g\in\Gamma$ of algebraic length $6$ satisfies,
$$\rho (x_0, gx_0)\leq (L-\eta )l_\Sigma (g)\,.$$ Hence, one 
obtain that every element $\gamma\in \Gamma$, satisfies,
$$\rho (x_0, \gamma x_0)\leq (L-\eta )(l_\Sigma (\gamma )-5) +5L\,.$$
Therefore we get from theorem \ref{theo} that 
$\Ent_{\Sigma} G \geq \eta= 10^{-3}\Big(1- \frac{\cosh({L\over4})}{\cosh({L\over2})}\Big)^4\,.$
\end{proof}

We may now prove the main theorem which we recall below,
\begin{theo}[Main theorem]
Let $(X,g)$ be a Cartan-Hadamard manifold whose sectional curvature satisfies $-a^2\leq K_g\leq -1$. Let $\Gamma$ be a discrete and finitely generated subgroup of the isometry group of $(X,g)$, then either $\Gamma$ is virtually nilpotent or its algebraic entropy is bounded below by an explicit constant $C(n,a)$.
\end{theo}
\begin{remark}
The constant is
$$C(n,a)=
\min \Bigg [ {\log (\ch ( {\mu (n,a)\over 4}))
\over 5+\log (\ch ({\mu (n,a)\over 4}) )}{\log 2\over 12}, {1\over 2000}\Big (1-{\ch ({\mu (n,a)\over 4})\over 
\ch ({\mu (n,a)\over 2})}\Big )^4,
{\mu (n,a)\over 4},$$
$$
{1\over 4}\big(1-\frac{1}{\cosh\mu(n,a)}\Big)^2,
{1\over 2}\log \Bigg({1\over 2}\Big(\cosh{\mu(n,a)\over 2}+\frac{1}{\cosh{\mu(n,a)\over 2}}\Big)\Bigg)  
\Bigg]
\,.$$
\end{remark}
\begin{proof}
If $S=\{ \sigma_1, \dots, \sigma_p\}$ is a finite generating set of $\Gamma$, proposition \ref{not-nilpotent} allows to reduce to the following three cases,

i) there exist $\sigma_i, \sigma_j\in S$ such that $L(<\sigma_i, \sigma_j>)\geq \mu (n,a)$ and such that the subgroup $<\sigma_i, \sigma_j >$ is not virtually nilpotent.

ii) There exist $\sigma_i, \sigma_j,\sigma_k \in S$ such that $L(<\sigma_i\sigma_j , \sigma_k>)\geq \mu (n,a)$ and such that $<\sigma_i\sigma_j , \sigma_k>$ is not virtually nilpotent.

iii) All $\sigma_i$'s are elliptic and, for all $i\ne j$, the subgroup $<\sigma_i, \sigma_j>$ fixes a point $y\in X$ or a point $\theta\in \partial X$.

In the first case (resp. the second case) proposition \ref{entropy-bounded} gives a lower bound of the algebraic entropy of $<\sigma_i, \sigma_j >$  (resp. $<\sigma_i\sigma_j , \sigma_k>$) with respect to the generating set $\{\sigma_i, \sigma_j \}$ (resp.  $\{\sigma_i\sigma_j , \sigma_k \}$), by the number,
$$\min\Bigg [ {\log (\ch ( {\mu (n,a)\over 4}))
\over 5+\log (\ch ({\mu (n,a)\over 4}) )}{\log 2\over 6}, {1\over 1000}\Big (1-{\ch ({\mu (n,a)\over 4})\over 
\ch ({\mu (n,a)\over 2})}\Big )^4  \Bigg ]\,,$$
using the fact that $L(\sigma_i, \sigma_j)\geq \mu (n,a)$, 
(resp. $L(\sigma_i\sigma_j , \sigma_k)\geq \mu (n,a)$).
We conclude in the two first cases (i) and (ii) by noticing that the entropy of $\Gamma$ with respect to $S$ is bounded below by $\Ent_{\{\sigma_i, \sigma_j  \}} (<\sigma_i, \sigma_j>)$ (resp. by ${1\over 2}\Ent_{\{\sigma_i\sigma_j , \sigma_k  \}} (<\sigma_i\sigma_j , \sigma_k>)$), since $d_{\{\sigma_i, \sigma_j  \}} \geq d_S$ (resp. $d_{\{\sigma_i\sigma_j, \sigma_k  \}} \geq {1\over 2} d_S$).

In the third case, proposition \ref{elliptic}, implies that,
$$\rho (x_0, \gamma x_0)\leq (L(S)-\eta /2)(l_S(\gamma )-1)+ L(S)\,,$$
where $\eta$ is given in proposition \ref{elliptic}. We conclude by applying theorem \ref{theo}, which gives $\Ent_S(\Gamma )\geq \eta$, and then bounding below $\eta$ using $L(S)\geq \mu (n,a)$.
\end{proof}

\section{Appendix}

In this section $(X,g)$ is a Cartan-Hadamard manifold of sectional curvature
$K\leq -1$. It is well known that if $\alpha$, $\beta$ are two hyperbolic isometries 
of $(X,g)$ with disjoint axes, then a sufficiently large power 
$\alpha ^N$ and $\beta ^N$ of $\alpha$  and $\beta$ generates a non abelian free group of $\Isom X$.
In \cite{cham-gui}, \cite{de}, it was shown that if $\Gamma$ is an hyperbolic group then
$N$ can be chosen independantly of $\alpha$ and $\beta$ in $\Gamma$ and under the same 
assumptions the $N$ was shown to depend only on the number of 
generators and the constant of hyperbolicity of $\Gamma$  \cite{cham-gui}.
In what follows we show that $N=N(\delta)$ can be chosen independantly of $\alpha$ and $\beta$
two hyperbolic isometries 
of $(X,g)$ with disjoint set of fixed points and displacement greater than or equal to 
a positive number $\delta$.

\begin{proposition}\label{ping-pong}
Let $(X,g)$ be a Cartan Hadamard manifold of sectional curvature $K\leq -1$
and $\Gamma$ a discrete group of isometries in $Isom(X,g)$. We assume 
that $\alpha$ and $\beta$ have disjoint set of fixed point and their
displacement satisfy $\l(\alpha) \geq \delta$ and  $\l(\beta) \geq \delta$,
where $\delta$ is a positive number. Then, $(\alpha ^N, \beta ^N)$ or
$(\alpha ^N, \beta ^{-N})$ generates a non abelian free semi-group,
where $N = E(\frac{5}{\delta})+1$.
\end{proposition}

Before going to the proof of the proposition \ref{ping-pong}
let us set some notations.
Let us write $x= x(t)$ and $y=y(t)$, $t\in \mathbb R$, the axes
of $\alpha$ and $\beta$. The points $\theta ^{\pm}= \lim_{t\to {\pm \infty}} x(t)$ and 
$\zeta ^{\pm}= \lim_{t\to {\pm} \infty} y(t)$ are the fixed points of $\alpha$ and $\beta$
on the ideal boundary $\partial X$ of $X$. Let us denote $x^+$ and $x^-$ the projections of 
$\zeta ^+$ and $\zeta ^-$ on the axis of $\alpha$. We can assume that $x^+$ is closer to $\theta ^+$
than $x^-$, (if it is not the case, we replace $\beta$ by $\beta ^{-1}$). 
Let us also denote $y_0$ the projection of $x^+$ on the axis of $\beta$.
We now parametrize $x$ and $y$ in such a way that $x(0)= x^+$ and $y(0)= y_0$. 
We set $t_1= N l(\alpha)= l(\alpha ^N)$ and $t_2= N l(\beta)= l(\beta ^N)$, where
$N = E(\frac{5}{\delta})+1$ is chosen as in the proposition.
We define $U ^{\pm}$ as the set of points $p\in X$ such that 
$\rho(p, x(\pm t_{1})) \leq \rho(p,x(o))$.
In the same way we define $V ^{\pm}$ as the set of points $p\in X$ such that 
$\rho(p, y(\pm  t_{2})) \leq \rho(p,y(o))$.
For a unit tangent vector $u\in T_x X$ at a point $x\in X$ and $\alpha \in [0, \pi[$
we denote 
$\mathcal{C} (u,\alpha)= \left\{ exp_x v \,: v\in T_{x} X \,,\angle(u,v) \in [0,\alpha[\,  \right\}$ 
the cone of angle $\alpha$ at $x$, where $exp_x$ is the exponential map at $x$.

We further need the following geometric lemmas.
For a triangle $ABC$ in $(X,g)$, we will write $\hat A$ the angle at $A$,
and $a$, $b$, $c$ the length of the sides opposite to $A$, $B$ and $C$.

\begin{lemma}\label{a1}
Let $ABC$ be a triangle in $(X,g)$ such that $\frac{\pi}{6} \leq \hat A \leq \pi$,
then $\rho (B,C) > \rho(A,B) + \rho (A,C) -4$. Moreover, if $\hat A \geq \frac{\pi}{2}$,
then $\rho (B,C) > \rho(A,B) + \rho (A,C) -1$.
\end{lemma} 
\begin{proof}
Since the curvature $K\leq -1$, we have
\begin{equation}\label{distance-comparison}
\cosh a \geq \cosh b\, \cosh c\, -\cos \hat{A}\, \sinh b \,\sinh c .
\end{equation}
The first inequality of lemma \ref{a1} will therefore be a consequence of the fact that
if $b+c >4$ then
\begin{equation}
\cosh (b+c-4) - \cosh b\, \cosh c\, + \cos \hat{A}\, \sinh b \,\sinh c <0.
\end{equation}
Setting $X = e^{-(b+c)}$ we have
$$
\cosh (b+c-4) - \cosh b\, \cosh c\, + \cos \hat{A}\, \sinh b \,\sinh c=
$$
$$
e^{(b+c)}\Big[(2e^4 -1 +\cos \hat{A})X^2 -(e^{-2b}+ e^{-2c})(1+\cos \hat{A})
-(1-\cos \hat{A}-2e^{-4}) \Big].
$$
Since $e^{-2b}+ e^{-2c} \geq 2e^{-(b+c)}$, we then get
$$
\cosh (b+c-4) - \cosh b\, \cosh c\, + \cos \hat{A}\, \sinh b \,\sinh c \leq e^{(b+c)}P(X)
$$
where 
$$
P(X)= (2e^4 -1 +\cos \hat{A})X^2 -(1+\cos \hat{A})X
-(1-\cos \hat{A}-2e^{-4})
$$
and $P(X)$ is negative when $P(0) <0$ and $P(e^{-4})<0$ which is the case 
if $\cos \hat{A} < 1-2e^{-4}$ and so when $\hat A \geq \pi/6$.
This proves the first inequality of the lemma. The second inequality 
is proved similarly when $\cos \hat{A} < 1-2e^{-1}$.
\end{proof}
 
\begin{lemma}\label{a2}
The sets $U^+$ and $U^-$ are contained in $\mathcal{C} (\dot{x}(0), \pi /6)$
and $\mathcal{C} (-\dot{x}(0), \pi /6)$ respectively.
\end{lemma}
\begin{proof}
We recall that $x(0) = x^+$. Let $c(t)$ be a geodesic ray starting at $x^+$ such that 
$\angle (\dot{x}(0), \dot{c}(0)) \geq \pi/6$. Since $t_1 \geq 5$, the lemma \ref{a1}
implies
$$
\rho(c(t), x(t_1)) > \rho(x^+, c(t)) + \rho(x^+, x(t_1)) -4 \geq \rho(c(t), x^+) \,
$$
therefore $c(t) \notin U^+$. The same argument holds for $U^-$. 
\end{proof}
Let us denote by $z_t$ the geodesic joining $x^+$ and $y(t)$ and $z_{\pm \infty}$
the geodesic joining $x^+$ and $y(\pm \infty)=\zeta^{\pm \infty}$.
\begin{lemma}\label{a3} 
The set $V^{\pm}$ is contained in $\mathcal{C} (\dot{z}_{\pm \infty}(0), \pi/3)$. 
\end{lemma}
\begin{proof}
Let us recall that the angle at $y(0) =y_0$ between $z_0$ and $y$ is equal to $\pi/2$, so that
the lemma \ref{a1} says that 
\begin{equation}\label{estim1}
\rm length (z_t) > length (z_0) + t -1\,,
\end{equation}
and in particular,
\begin{equation}\label{estim2}
\rm length (z_{t_2}) > length (z_0) + t_{2} -1\,.
\end{equation}
Let us now show that $\angle (\dot{z}_{t_2} (0), \dot{z}_{+\infty}(0)) \leq \pi/6$.
Assume by contradiction that $\angle (\dot{z}_{t_2} (0), \dot{z}_{+\infty}(0)) > \pi/6$,
then by lemma \ref{a1} we have, when $t$ tends to $+\infty$,
\begin{equation}\label{estim3}
 t-t_{2} > \rm length (z_{t_2}) +length (z_t) -4\,,
\end{equation}
but summing up (\ref{estim1}) and (\ref{estim2}), in (\ref{estim3}) leads to a contradiction
since $t_2 \geq 5$.
Therefore we have $\angle (\dot{z}_{t_2} (0), \dot{z}_{+\infty}(0)) \leq \pi/6$.
Let now consider a geodesic ray $c$ starting at $x^+$ such that
$\angle (\dot{c}(0), \dot{z}_{+\infty}(0)) \geq \pi/3$. Thus, 
$\angle (\dot{c}(0), \dot{z}_{t_2}(0)) \geq \pi/6$ and by lemma \ref{a1} we get
$$
\rho (c(t), y(t_2)) > \rho(c(t), x^+) + \rm length (z_{t_2}) -4\,,
$$
and applying again the lemma \ref{a1},
$$
\rho (c(t), y(t_2)) > \rho(c(t), x^+) + \rm length (z_0) + t_2 -5\,.
$$
The last inequality becomes by triangle inequality,
$$
\rho (c(t), y(t_2) ) > \rho(c(t), y_0) + t_2 -5\,,
$$
therefore
$\rho (c(t), y(t_2) ) > \rho(c(t), y_0)$ since $t_2  \geq 5$.

We have proved that a geodesic ray $c$ starting at $x^+$ such that
$\angle (\dot{c}(0), \dot{z}_{+\infty}(0)) \geq \pi/3$ does not intersect $V^+$.
This proves that $V^{+} \subset \mathcal{C} (\dot{z}_{+ \infty}(0), \pi/3)$. 
By the same argument we also have $V^{-} \subset \mathcal{C} (\dot{z}_{- \infty}(0), \pi/3)$,
which ends the proof of the lemma. 
\end{proof}

\begin{lemma}\label{a4}
We have $U^{\pm} \cap V^{\pm} = \emptyset$,  $U^{+} \cap U^{-} = \emptyset$,  
$V^{+} \cap V^{-} = \emptyset$.
\end{lemma}
\begin{proof}
From the angle relations, $\angle (\dot{x} (0), \dot{z}_{+\infty}(0)) = \pi/2$,
$\angle (\dot{x} (0), \dot{z}_{-\infty}(0)) \geq \pi/2$,  
$\angle (\dot{x} (0),-\dot{x} (0) ) = \pi$, and the relative position of $x^+$, $x^-$ and $\theta ^+$,
it follows that
$\mathcal{C}(\dot{x} (0), \pi/6)$ does not intersect 
$\mathcal{C}(\dot{z}_{+\infty} (0), \pi/3)$, $\mathcal{C}(-\dot{x} (0), \pi/6)$ and 
$\mathcal{C}(\dot{z}_{-\infty} (0), \pi/3)$. Therefore by the lemmas \ref{a2}, \ref{a3} we  
conclude that $U^+$ does not intersect $U^-$, $V^+$ and  $V^-$.
Now since $\angle (-\dot{x} (0),\dot{z}_{+\infty} (0)) = \pi/2$, we have
$\mathcal{C}(\dot{z}_{+\infty} (0), \pi/3) \cap \mathcal{C}(-\dot{x} (0), \pi/6) = \emptyset$,
hence $V^+ \cap U^- = \emptyset$. If $p\in V^+ \cap V^-$, we have
$\rho (p, y(0)) \geq \rho (p, y(-t_2))$ and $\rho (p, y(0)) \geq \rho (p, y(t_2))$
which contradicts the convexity of the function $t\to \rho (x(t), p)$. Therefore 
$V^+ \cap V^- = \emptyset$.
\end{proof}

\begin{lemma}\label{a5}
We have $\alpha ^N (V^+) \subset U^+$ and  $\beta ^N (U^+) \subset V^+$. 
\end{lemma}
\begin{proof}
Since $x$ and $y$ are the axes of $\alpha ^N$ and $\beta ^N$ respectively 
we have $\alpha ^N (x(-t_1)) = x(0)$,  $\beta ^N (y(-t_1)) = y(0)$,
$\alpha ^N (x(0)) = x(t_1)$ and $\beta ^N (y(0)) = y(t_2)$. Therefore
for any $p\in X-U^-$, we have $\alpha ^N (p) \in U^+$ and similarly
for any $p\in X-V^-$, we have $\beta ^N (p) \in V^+$
by definition of $N$. On the other hand, by the lemma \ref{a4}, we have 
$V^+ \subset X-U^-$ and $ U^+ \subset X-V^-$, which concludes.  

\end{proof}

The proof of the proposition \ref{ping-pong} is a direct application of the lemma \ref{a5}
by a standard ping-pong argument.

\end{document}